\documentclass[11pt,openany,twoside]{amsart}
\usepackage{graphicx}
\usepackage{stmaryrd}
\usepackage{amssymb}
\usepackage{amsthm}
\usepackage{dsfont}
\usepackage{hyperref}
\usepackage{mathrsfs}
\usepackage{stmaryrd}
\usepackage[lite,initials,msc-links]{amsrefs}
\usepackage{amsmath}
\usepackage{centernot}
\usepackage{enumerate}
\usepackage{verbatim}
\usepackage{dsfont}
\usepackage{bm}
\usepackage{geometry}
\usepackage{faktor}
\usepackage{relsize}
\usepackage{centernot}
\usepackage{enumerate}
\usepackage{verbatim}
\usepackage{mathtools}
\usepackage{color,soul}
\usepackage{centernot}
\usepackage[mathscr]{euscript}
\usepackage{caption}

\newtheorem{theorem}{Theorem}[section]
\newtheorem{lemma}[theorem]{Lemma}
\newtheorem{proposition}[theorem]{Proposition}
\newtheorem{corollary}[theorem]{Corollary}
\numberwithin{equation}{section}

\theoremstyle{definition}

\theoremstyle{remark}
\newtheorem{remark}{Remark} 
\theoremstyle{question}

\newcommand{\bl}{\bullet}
\newcommand{\wi}{\circ}

\newcommand{\n}{n}

\newcommand{\T}{\mathbb T}
\newcommand{\OO}{\mathcal O}
\newcommand{\LL}{\mathcal L}
\newcommand{\Space}{\Omega}
\newcommand{\wind}{\textnormal{w}}

\title[On delocalization in the six-vertex model]{On delocalization in the six-vertex model}

\date{\today}

\author{Marcin Lis}{
\address{Faculty of Mathematics\\ University of Vienna \\
Oskar-Morgenstern-Platz 1\\
1090 Wien}
\email{marcin.lis@univie.ac.at}}

\begin{document}

\begin{abstract} 
We show that the six-vertex model with parameter $c\in[\sqrt 3, 2]$ on a square lattice torus has an ergodic infinite-volume limit as the size of the torus grows to infinity.
Moreover we prove that for $c\in[\sqrt{2+\sqrt 2}, 2]$, the associated height function on $\mathbb Z^2$ has unbounded variance.

The proof relies on an extension of the Baxter--Kelland--Wu representation of the six-vertex model 
to multi-point correlation functions of the associated spin model. 
Other crucial ingredients are the uniqueness and percolation properties of the critical random cluster measure for $q\in[1,4]$, and recent results relating the decay of correlations in the spin model with the delocalization of the height function.
\end{abstract}

\maketitle
\section{Introduction}

\subsection*{Background and main results}
An \emph{arrow configuration} on a $4$-regular graph is an assignment of an arrow to every edge such that exactly two arrows point towards each vertex. 
The \emph{six-vertex model} (or more precisely the \emph{F model}) with parameter $c>0$ on a finite 4-regular graph embedded in a surface is
the probability measure on all arrow configurations that is proportional to $c^{N}$, where $N$ is the number of vertices of type $3a$ or $3b$ in the configuration (see Fig.~\ref{fig:6v}). 
These are the vertices for which the arrows alternate between incoming and outgoing as one goes around the vertex. 

The three-dimensional prototype of the model with $c=1$ (the uniform measure on arrow configurations) was introduced by Pauling~\cite{Pau} in 1935 to study the residual entropy of ice arising from the phenomenon of hydrogen bonding. 
The square lattice version discussed here, called the F model, first appeared in the work of
Rys on antiferroelectricity~\cite{Rys}. The exact value of the free energy per site on the square lattice
was given by Lieb~\cite{Lieb,Lieb1} using the transfer matrix method. Since then the six-vertex model has been a prominent example of an integrable lattice model 
of equilibrium statistical mechanics.
For a detailed account of the model and its history we refer the reader to~\cite{LiebWu,BaxterBook,Reshetikhin}.

\begin{figure}
\begin{center}
 \includegraphics[scale=0.9]{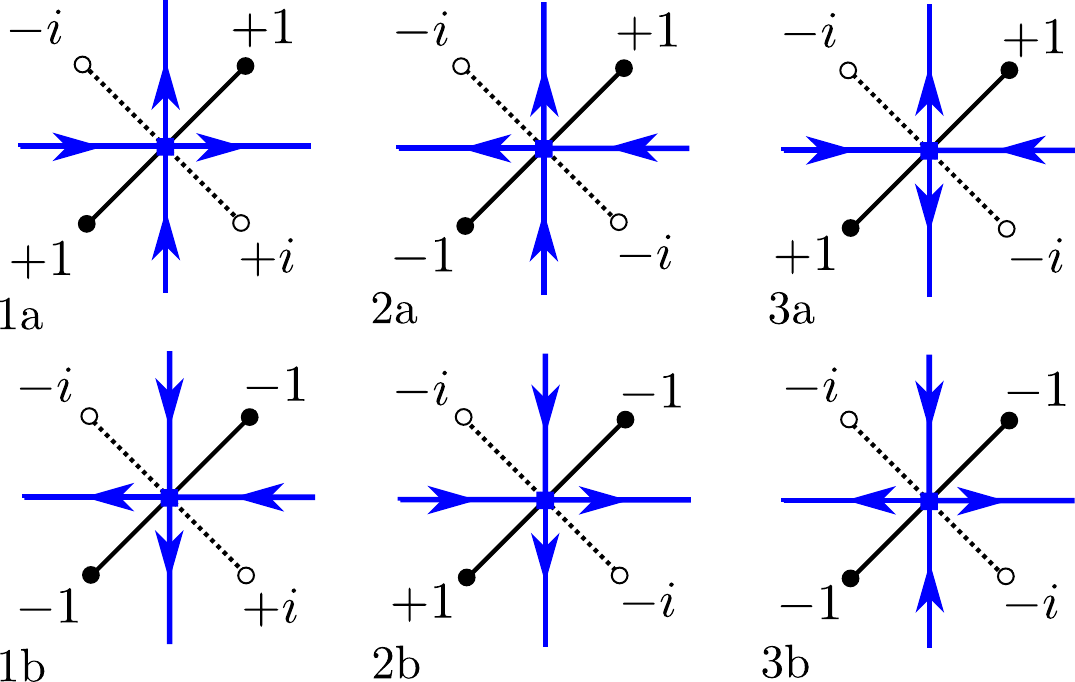}
\caption{Six possible arrow configurations around a vertex of~$\T_{\n}$. To each such configuration there correspond two spin configurations which differ by a global sign change.
The configuration with the upper left spin fixed to $-i$ is depicted here.
The solid (resp.\ dotted) black line represents an edge of $\T^{\bl}_{\n}$ (resp.\ $\T^{\wi}_{\n}$). 
}
 \label{fig:6v}
\end{center}
\end{figure}

In this article we consider the six-vertex model on a toroidal piece of the square lattice 
\[
\mathbb T_{\n}=(\mathbb Z/2nj_1 \mathbb Z)\times (\mathbb Z/2nj_2 \mathbb Z)
\] 
of size $2nj_1 \times 2nj_2$, where $j_1,j_2 \in \mathbb N$ are fixed and $\n$ increases to infinity. We denote the corresponding probability measure by $\mu_\n=\mu_\n^c$.
In the first main result we establish convergence to an infinite-volume measure for the model with $c\in[\sqrt{3},2]$, and derive a spatial mixing property of the limit.

\begin{theorem}\label{thm:existence}
Let $c\in[\sqrt{3},2]$.
\begin{itemize}
\item[$(i)$] There exists a translation invariant probability measure $\mu=\mu^c$ (independent of $j_1$ and $j_2$) on arrow configurations on $\mathbb Z^2$, 
such that 
\[
\mu_{\n} \to \mu \text{ weakly} \quad \textnormal{ as } n \to \infty. 
\]
\item[$(ii)$] There exists $\kappa=\kappa(c)>0$ such that for any two local events $A$ and $B$ depending on the state (orientation) of edges in finite boxes 
$\Lambda, \Lambda' \subset E(\mathbb Z^2)$ respectively, we have
\begin{align}\label{eq:mixing}
|\mu(A\cap B) -\mu(A)\mu(B)| \leq K_{|\Lambda|,|\Lambda'|}d(\Lambda,\Lambda')^{-\kappa},
\end{align}
where $d(\Lambda,\Lambda')$ is the graph distance between $\Lambda$ and~$\Lambda'$, and where $K_{|\Lambda|,|\Lambda'|}$ depends only on the size of the boxes.
\item[$(iii)$] In particular, $\mu$ is ergodic with respect to any nontrivial translation.
\end{itemize}
\end{theorem}

An observable of interest in the six-vertex model is its \emph{height function} $h$. For now we consider it directly in the infinite volume limit as an integer-valued function defined on the faces of $\mathbb Z^2$.
We first chose a chessboard white and black coloring of the faces of $\mathbb Z^2$.
For reasons to become clear later, we set $h$ to be $\pm 1$ with probability $1/2$ on a chosen white face $u_0$ next to the origin of $\mathbb Z^2$. For any other face $u$ 
and a dual oriented path $\gamma$ connecting $u_0$ with $u$,
we denote by $h^{\gamma}_{\leftarrow}(u)$ and $h^{\gamma}_{\rightarrow}(u)$ the numbers of arrows in the 
underlying six-vertex configuration that cross $\gamma$ from right to left, and from left to right respectively.
The height at $u$ is then defined by
\begin{align} \label{eq:hf}
h(u)-h(u_0) = h^{\gamma}_{\leftarrow}(u) - h^{\gamma}_{\rightarrow}(u).
\end{align}
That the right-hand side is independent of $\gamma$ follows from the fact that $\mathbb Z^2$ is simply connected and from the property that six-vertex configurations form conservative flows.

It is predicted that the model should undergo a phase transition at $c=2$ in the sense that the variance of the height-function should be uniformly bounded 
(over all faces of $\mathbb Z^2$) for $c>2$ (the \emph{localized} regime), and should be unbounded for $c\leq 2$ (the \emph{delocalized} regime). 
So far this has been rigorously confirmed 
for $c>2$~\cite{Disc,GlaPel}, $c=2$~\cite{DCST,GlaPel}, $c=1$~\cite{She,CPST,LogVar}, the \emph{free fermion} point $c=\sqrt{2}$~\cite{Ken01} corresponding to the dimer model,
and a small neighborhood of $c=\sqrt{2}$~\cite{GMT}.
Moreover, logarithmic (in the distance to the origin) divergence of the variance was established in~\cite{DCST,GlaPel,LogVar}.
We note that a closely related result was recently proved also in the model of uniform Lipschitz functions on the triangular lattice~\cite{GlaMan}.
Finally, a much stronger property was obtained in~\cite{Ken01,GMT}, namely that the 
fluctuations of the height function in the scaling limit are described by the Gaussian free field. 
The following result adds to this list by identifying delocalization in the weak sense for all $c\in[\sqrt{2+\sqrt{2}},2]$.
\begin{theorem}[Delocalization of the height function] \label{thm:delocalization}
Let $c\in[\sqrt{2+\sqrt{2}},2]$. Then under the infinite volume measure~$\mu$ we have
\begin{align} \label{eq:delocalization}
\mathbf{Var}_{\mu} [h(u)] \to \infty \quad \text{ as } \quad |u| \to \infty,
\end{align}
where $u$ is a face of $\mathbb Z^2$. 
\end{theorem}

We note that Theorem~\ref{thm:existence} and a stronger version of Theorem~\ref{thm:delocalization} yielding logarithmic 
divergence of the variance have been independently proved for all $c\in[1,2]$ in the work
of Duminil-Copin et al.~\cite{DKMO}. However, the methods that we use are different than those of~\cite{DKMO} and arguably more elementary. 
We also believe that the ideas presented in this paper will be useful in further analysis of 
the six-vertex model, in particular in questions regarding its scaling limit.

\subsection*{Outline of the approach}
Before explaining the arguments in detail, we give a brief overview of our approach.
We color the faces of $\T_{\n}$ and $\mathbb Z^2$ in a checkerboard manner. 
Let $\mathbb Z^2_{\bl}$ (resp.\ $\mathbb Z^2_{\wi}$) be the square lattice of side length $\sqrt2 $ rotated by $\frac{\pi}4$ whose vertices are the black (resp.\ white) faces of $\mathbb Z^2$, and 
where two vertices are adjacent if the corresponding faces of $\mathbb Z^2$ share a vertex (see Fig.~\ref{fig:BKW1}).
Let $\T_{\n}^{\bl}$ and $\T_{\n}^{\wi}$ be defined analogously for $\T_{\n}$.
The first main ingredient of the proofs of both main theorems is the Baxter--Kelland--Wu correspondence~\cite{BKW} between the six-vertex model on $\T_{\n}$ with parameter $c\in[\sqrt{3},2]$ 
and the critical random cluster model on $\T_{\n}^{\bl}$ 
with cluster parameter 
\begin{align} \label{eq:qc}
q=(c^2-2)^2\in [1,4].
\end{align}
For $c\in[\sqrt{3},2)$, unlike for $c\in [2,\infty)$,
this representation is not a stochastic, but rather a complex-measure coupling between the six-vertex and the (slightly modified) critical
random cluster model. For this reason it has not been clear how to transfer relevant probabilistic information between the two sides of this coupling. 
The main novelty of our approach is an extension of this correspondence to identities between correlation functions of certain observables.
These observables on the side of the six-vertex model are simply spins assigned to the faces of the lattice and given by 
\begin{align} \label{eq:sigmadef}
\sigma(u)=i^{h(u)},
\end{align}
where $i$ is the imaginary unit. Note that from a spin configuration one recovers the six-vertex configuration in a unique (and local) way, and hence the spins 
carry all the probabilistic information of the six-vertex model. 
Recall that our convention is to fix the height function to be $\pm1$ with equal probability on a white face $u_0$ adjacent to the origin. This makes the distribution of spins invariant under 
the sign change $\sigma \mapsto -\sigma$. 
Since the parity of the height function always changes between adjacent faces
the spins are real on the black, and imaginary on the white faces of $\mathbb Z^2$. 
Also note that $\sigma$ is always well-defined locally. 
Globally however, the height function may have a non-trivial period when one goes around the torus.
If the period is nontrivial mod $4$, the spin picks up a multiplicative term of $-1$ when going around the torus. This is a technical inconvenience that we discuss in more detail in 
the following sections.

To illustrate the type of identities between correlation functions obtained in this paper we briefly discuss here the simplest case of the two-point function.
Let~$\phi$ be the critical random cluster measure on the rotated lattice $\mathbb Z^2_{\bl}$ (see Fig.~\ref{fig:large}) with cluster parameter $q\in[1,4]$ as in \eqref{eq:qc}. The fact that this measure is unique, which was 
established by Duminil-Copin, Sidoravicius and Tassion in \cite{DCST}, is crucial and constitutes the second main ingredient of our proof of Theorem~\ref{thm:delocalization}. 
For two black faces $u,u' \in \mathbb {Z}^2_{\bl}$, we establish that
\begin{align} \label{eq:example}
\mathbf E_{\mu}[\sigma(u)\sigma(u')] = \mathbf E_{\phi}\big [\rho^{N(u,u')}(-\rho)^{N(u',u)}\big],
\end{align}
where 
\begin{align} \label{eq:lambda}
\rho = \tan \lambda, \qquad \text{ with } \qquad \sqrt q = 2 \cos \lambda,
\end{align}
and where $N(u,u')$ is the number of loops on $\mathbb Z^2$ that disconnect $u$ from $u'$ in the loop representation of the random cluster model. 
Analogous identities to \eqref{eq:example} for many-point correlation functions already on the finite level of 
$\T_{\n}$ are also established, and are the main tool to obtain the existence of the infinite-volume measure, and hence prove Theorem~\ref{thm:existence}.

To show delocalization of the height function as stated in Theorem~\ref{thm:delocalization}, we use a recent result of the author~\cite{Lis19}. We first establish decorrelation of spins saying that
\begin{align} \label{eq:decorrelation}
\mathbf E_{\mu}[\sigma(u)\sigma(u')] \to 0, \qquad \text{ as } \qquad |u-u'|\to \infty.
\end{align}
This in turn implies (up to technical details that are taken care of in the present article) that there is no percolation in the associated percolation model studied in~\cite{Lis19,GlaPel}. Finally, non-percolation was shown in~\cite{Lis19} to imply delocalization.

It is clear that identity \eqref{eq:example} is useful for proving such decorrelation of spins \eqref{eq:decorrelation}. Indeed, for $c\in (\sqrt{2+\sqrt 2},2]$, we have that $0\leq \rho<1$.
On the other hand, by the results of \cite{DCST} we know that $N(u,u')\to \infty$ as $|u-u'|\to \infty$ $\phi$-almost surely. Hence, by~\eqref{eq:example} the spins decorrelate and the height function delocalizes.
The peripheral case $c=\sqrt{2+\sqrt 2}$ corresponding to $q=2$ and $\rho =1$ requires a slightly different argument.

To finish this discussion we note that Theorem~1.1 implies decorrelation of local increments of the height function for $c\in[\sqrt 3,2]$. 
To prove delocalization however, we need additional information on the global 
increment between two far-away points $u,u' \in \mathbb {Z}^2_{\bl}$. 
From the results of \cite{Lis19}, it turns out that the sufficient information is the behaviour of the parity of $({h(u)-h(u')})/2$. This is provided by~\eqref{eq:decorrelation} since for a black face $u'$, $h(u')$ is even, and hence
\[
\sigma(u)\sigma(u')= (-1)^{({h(u)-h(u')})/2}.
\]

A natural question that remains is if our approach to prove delocalization extends to the case $\rho =\tan \lambda>1$, or equivalently $c\in [\sqrt 3, \sqrt{2+\sqrt 2})$, (by possibly studying a different observable than \eqref{eq:decorrelation} to obtain delocalization).
In the rest of the paper we provide the precise statements and the remaining necessary details for the proofs of our results.

\subsection*{Acknowledgments} I would like to thank Nathana\"{e}l Berestycki for stimulating discussions and his insight into Lemma~\ref{lem:loopcorrelationslimit}, and Alexander Glazman for 
many valuable discussions.
\begin{figure}
\begin{center}
 \includegraphics[scale=0.9]{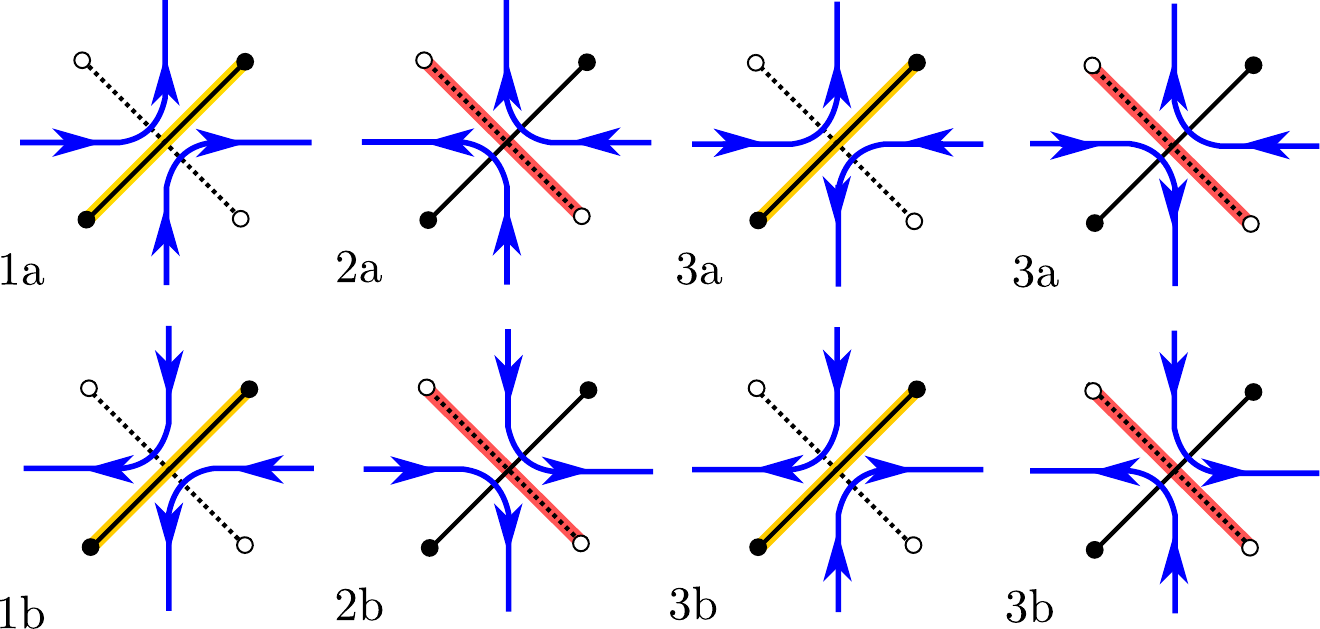}
\caption{Eight possible ways the loops connect at a vertex.
The yellow and red edges are the open edges in the percolation configuration $\xi$ and $\xi^{\dagger}$ respectively. }
 \label{fig:BKW}
\end{center}
\end{figure}

\section{The Baxter--Kelland--Wu representation of spin correlations}
The Baxter--Kelland--Wu (BKW) correspondence is the starting point of our argument. 
We recall it here while simultaneously establishing closely related identities for correlations of the $\sigma$-spins.

The first step is to represent the arrow configurations on $\T_{\n}$ 
as \emph{fully packed} configurations of directed and noncrossing loops $\vec { L}$, also on $\T_{\n}$. The term fully packed means that each edge of $\T_{\n}$ is 
traversed exactly once by a loop from~$\vec{ L} $. The loops in $\vec L$ should follow the arrows of the arrow configuration~$\alpha$, and the only choice remaining is to
decide how the directed edges connect at each vertex to form noncrossing loops. 
For configurations of type 1 and 2, there is no choice, whereas for configurations of type 3 we can choose two different types of connections, 
see Fig.~\ref{fig:BKW}. On the other hand, to reverse the map in order to obtain $\alpha$ from $\vec L$, it is enough to keep the information about the orientation of each edge and otherwise forget how the loops connect at the vertices. 
This gives a many-to-one map from $\vec \LL_{\n}$, defined to be the set of all fully packed oriented loop configurations, to $\OO_{\n}$ -- the set of all arrow configurations on $\T_{\n}$.

The crucial idea now is to parametrize the six-vertex weights in terms of the types of turns the loops make at each vertex. To this end, we define the weight of an oriented loop configuration by
\begin{align} \label{eq:loopweight}
w(\vec L) = e^{\frac{i\lambda}4 (\textnormal{left}(\vec L)-\textnormal{right}(\vec L))},
\end{align}
where $\lambda$ is as in \eqref{eq:lambda}, and where $\textnormal{left}(\vec L)$ and $\textnormal{right}(\vec L)$ are the total numbers of left and right turns of all the loops in the configuration.
Note that at each vertex of type 1 or 2, the loops make turns in opposite directions and hence the joint contribution of these two turns to the weight is $1$.
On the other hand, for vertices of type 3, we either have two turns left or two turns right which yields a total weight $2\cos \tfrac{\lambda}2=c$.
This exactly means that after projecting the renormalized complex measure on 
$\vec \LL_{\n}$ induced from the weight~\eqref{eq:loopweight} onto arrow configurations~$\alpha$ (i.e., summing over all oriented loop configurations corresponding to $\alpha$) we recover the six-vertex probability measure $\mu_{\n}$.

The next step of the correspondence is to go from an oriented loop configuration $\vec L$ to an unoriented one by simply forgetting the orientations of the loops. 
To this end, note that after reorganizing the factors in \eqref{eq:loopweight} 
according to which turn is made by which loop, we obtain that
\begin{align} \label{eq:loopweight1}
w(\vec L) =   \prod_{\vec \ell\in \vec L}  e^{\frac{i\lambda}4 (\textnormal{left}(\vec \ell)-\textnormal{right}(\vec \ell))}= \prod_{\vec \ell\in \vec L} e^{{i\lambda}\wind(\vec \ell)},
\end{align}
where $\textnormal{left}(\vec \ell)$ and $\textnormal{right}(\vec \ell)$ are the total numbers of left and right turns of a single oriented loop $\vec\ell$, 
and where $\wind(\vec \ell)$ is the total \emph{winding number} of the loop. 
The important observation here is that if $\vec \ell$ is contractible on the underlying torus, then $\wind(\vec \ell)=\pm 1$ depending on the counterclockwise or clockwise orientation of the loop, 
and $\wind(\vec \ell)=0$
if $\vec \ell$ is noncontractible. By an unoriented loop configuration $L$ we mean a fully packed configuration of noncrossing loops obtained from some $\vec L\in \vec  \LL_{\n}$ 
by erasing all arrows from the edges.
From \eqref{eq:loopweight1} we can conclude that the weights $w(\vec L)$ induce a probability measure $\phi_\n$ on the set of fully-packed unoriented loop configurations $\LL_{\n}$ given by
\begin{align} \label{eq:nu}
\phi_\n( L) =  \frac{1}{Z_{\n}} \prod_{ \ell\in  L}  ( e^{{i\lambda}\wind(\vec \ell)}+ e^{-{i\lambda}\wind(\vec \ell)}) =  \frac{1}{Z_{\n}}\sqrt{q}^{|L|} \big(\tfrac{2}{\sqrt q}\big)^{|L_{\textnormal{nctr}}|},
\end{align}
where \[
Z_\n = \sum_{\alpha \in \mathcal O_\n} c^{N(\alpha)}
\]
is the partition function of the six-vertex model, and $L_{\textnormal{nctr}}$ is the set of noncontractible loops in $L$.

\begin{figure}
\begin{center}
 \includegraphics[scale=0.62]{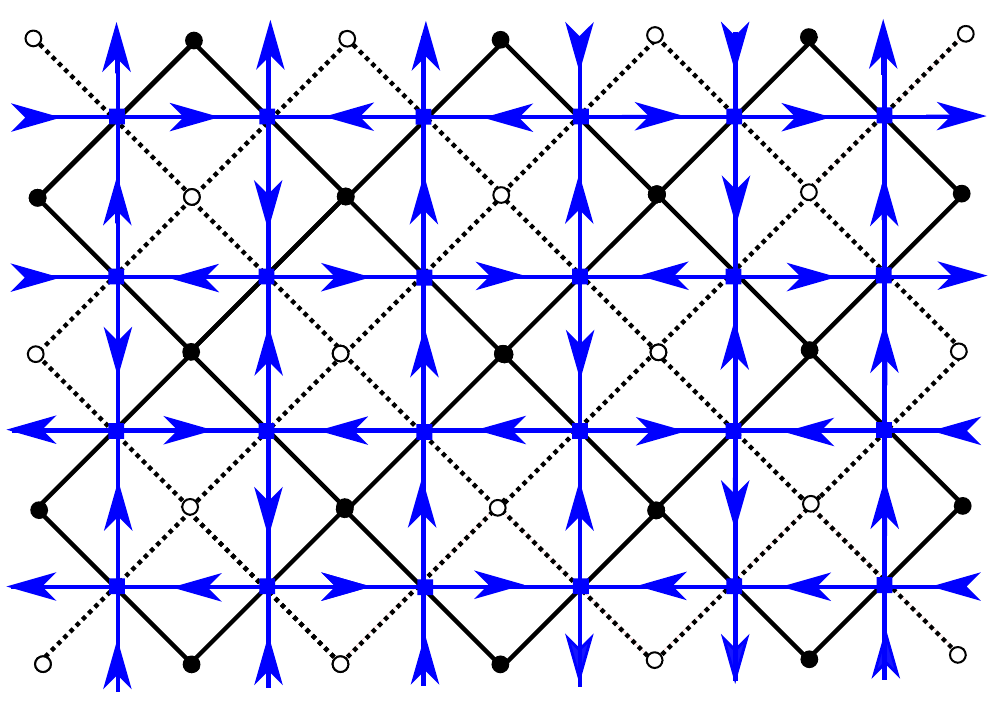} \hspace{0.1cm} \includegraphics[scale=0.62]{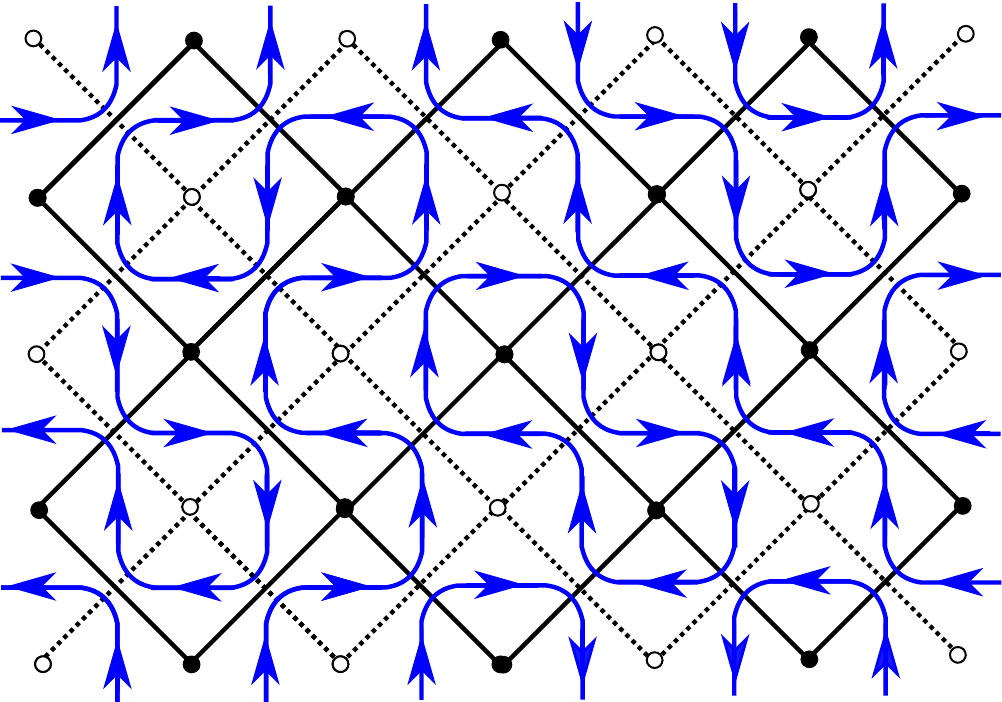} 
\caption{An arrow configuration on a $6\times 4$ torus and a corresponding fully packed configuration of directed loops }
 \label{fig:BKW1}
\end{center}
\end{figure}

Before discussing the connection with the random cluster model, let us derive the necessary formulas for the spin correlations as expectations of certain loop statistics under $\phi_{\n}$. 
As already mentioned, one needs to take slightly more care when defining the height function $h$ and hence the spins~$\sigma$
in finite volume as one may pick up a nontrivial period of $-1$ when going around the torus. To circumvent this obstacle, we identify the vertices of $\T_{\n}$ with those of the box
\begin{align} \label{eq:box}
\{-nj_1+1,\ldots,nj_1\}\times \{-nj_2+1,\ldots,nj_2\} \subset \mathbb Z^2.
\end{align}
As before the height function at $u_0$ is chosen to be $\pm 1$ with equal probability. For every other face, we use formula \eqref{eq:hf} with the 
restriction that the path $\gamma$ cannot take a step from a vertex with the $i$-th coordinate equal to $nj_i$ to a vertex with the same coordinate equal to $-nj_i+1$ and vice versa, for $i=1,2$.

Let $u_1,\ldots,u_{p}\in \mathbb T^{\bl}_\n$ and $v_1,\ldots,v_{r}  \in \mathbb T^{\wi}_\n$ be
black and white faces of $\T_{\n}$ respectively. 
We are interested in the the correlation function
\[
\mathbf E_{\mu_{\n}} \Big[\prod_{i=1}^{p}\sigma(u_i) \prod_{j=1}^{r}\sigma(v_j) \Big].
\] 
One can see that if either $p$ or $r$ is odd, then this expectation is zero by symmetry. Indeed, first note that after fixing spins on one sublattice and reversing all arrows, 
the spins on the other sublattice change sign.
Furthermore the six-vertex model is invariant under arrow reversal, and we also choose the distribution of $\sigma(u_0)$ to be symmetric.
Hence, we can assume that $p=2k$ and $r=2l$. 

We chose half of the faces $u_1,\ldots,u_{2k},v_1,\ldots,v_{2l}$ and declare them \emph{sources}, and we call the remaining half \emph{sinks}.
We now fix $k+l$ directed paths in the dual of $\T_{\n}$ that connect pairwise the sources to the sinks, and 
define $\Gamma$ to be the collection (sometimes called a \emph{zipper}) of directed edges of $\T_{\n}$ which cross these paths from right to left, see Fig.~\ref{fig:zipper}.
It is possible that one edge crosses multiple paths. We also define $s^{+}_i$ and $s^-_i$ to be the source and sink of path number $i$ respectively for $1\leq i\leq k+l$.

\begin{figure}
\begin{center}
 \includegraphics[scale=0.9]{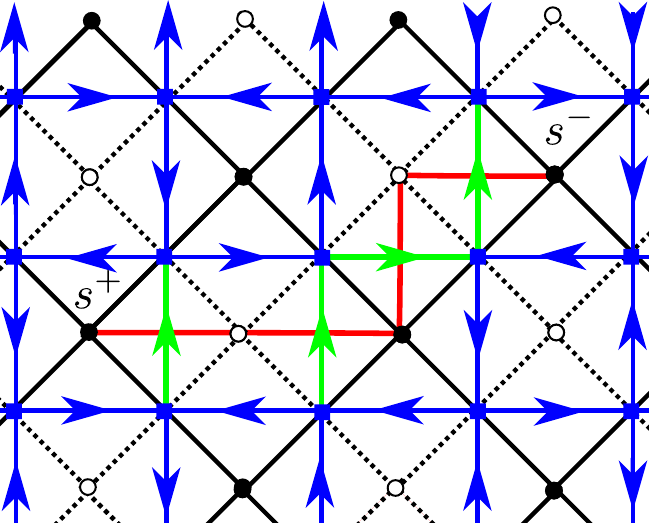} 
\caption{The directed path (red) with source $s^+$ and sink $s^-$. The green edges represent (up to arrow reversal) the zipper $\Gamma$ used in the computation of $\mathbf{ E}_{\mu_{\n} }[\sigma(s^+)\sigma(s^-)]$ 
(the actual edges in the zipper cross the path from right to left).
For this arrow configuration $\alpha$, we have $\epsilon(\alpha)=\sigma(s^+)\sigma(s^-)=-1$}
 \label{fig:zipper}
\end{center}
\end{figure}

Having fixed~$\Gamma$, for any other collection of directed edges $H$, we define 
\begin{align} \label{eq:epsilon}
 \epsilon(H)=i^{|\Gamma\cap H|}(-i)^{|\Gamma\cap (-H)|} \in \{ 1,i,-1,-i \},
\end{align}
where $-H$ is the set of all reversed edges from $H$. 
Note that 
\begin{align} \label{eq:reverse}
 \epsilon(H) \epsilon(-H)=1.
 \end{align}
Below, with a slight abuse of notation, we will identify six-vertex configurations~$\alpha$, 
oriented loop configurations $\vec L$, and single oriented loops $\vec \ell$ with the naturally associated collections of directed edges.
Recall that by our convention, $\sigma(u_i)^2=1$ and $\sigma(v_j)^2=-1$, or in other words
\[
i^{h(u_i)}=i^{-h(u_i)} \quad \textnormal{ and } \quad i^{h(v_j)}=-i^{-h(v_j)}.
\] 
Writing $s$ for the number of white sinks, we have
\begingroup
\allowdisplaybreaks
\begin{align}
(-1)^s \mathbf E_{\mu_{\n}} \Big[\prod_{i=1}^{2k}\sigma(u_i) \prod_{j=1}^{2l}\sigma(v_j) \Big]&= (-1)^s \mathbf E_{\mu_{\n}} \Big[i^{{\sum_{i=1}^{2k}h(u_i)+ \sum_{j=1}^{2l}h(v_j) }}\Big] \nonumber\\
&=  \mathbf E_{\mu_{\n}} \Big[i^{\sum_{i=1}^{k+l}(h(s^+_{i})-h(s^-_i)) }\Big]\nonumber \\
&=   \mathbf E_{\mu_{\n}} [\epsilon (\alpha)]\nonumber \\
&= \frac 1{Z_{\n}} \sum_{\vec L \in \vec \LL}  \epsilon(\vec L) w(\vec L)\nonumber\\
&=  \frac 1{Z_{\n}} \sum_{\vec L \in \vec \LL}  \prod_{\vec\ell \in \vec L} e^{{i\lambda}\wind(\vec \ell)} \epsilon (\vec\ell)\nonumber\\
&=  \frac 1{Z_{\n}} \sum_{ L \in  \LL} \big( \prod_{\ell \in  L} \rho(\ell)\big) \sqrt{q}^{|L|} \big(\tfrac{2}{\sqrt q}\big)^{|L_{\textnormal{nctr}}|} \nonumber,
\end{align}
where 
\begin{align} \label{def:rho}
\rho(\ell)=\frac{e^{{i\lambda}\wind(\vec \ell)} \epsilon (\vec\ell) +e^{{i\lambda}\wind(-\vec \ell)} \epsilon (-\vec\ell)}
{e^{{i\lambda}\wind(\vec \ell)} +e^{{i\lambda}\wind(-\vec \ell)}},
\end{align}
\endgroup
with $\vec \ell$ and $-\vec \ell$ being the two oriented versions of $\ell$. 
To get the third equality we represented each of the increments of the height function in the second line as a sum of one-step increments along the fixed paths, 
and then used the definitions of $h$~\eqref{eq:hf} and $\epsilon$~\eqref{eq:epsilon}. To obtain the fourth identity we followed the same reasoning as in the 
standard BKW representation. This can be done since the observable $\epsilon$ depends only on the orientations of the arrows, and this information is preserved when going form $\alpha$ to $\vec L$.
The last equality follows by forgetting the orientations of the loops as we did in \eqref{eq:nu}.

As a consequence we get the following crucial identity for correlation functions.
\begin{lemma} \label{lem:loopcorellations} Let $u_1,\ldots, u_{2k} \in \mathbb T^{\bl}_\n$ and $v_1,\ldots, v_{2l} \in \mathbb T^{\wi}_\n$ 
be black and white faces of $\mathbb T_\n$ respectively, and let $\rho$ is defined in~\eqref{def:rho} and $s$ is the number of white sinks. Then
\[
\mathbf E_{\mu_{\n}} \Big[\prod_{i=1}^{2k}\sigma(u_i) \prod_{j=1}^{2l}\sigma(v_j) \Big] = (-1)^{s}\mathbf E_{\phi_{\n}} \Big[\prod_{\ell\in L} \rho(\ell)\Big].
\]
\end{lemma}
We note for future reference that the same formula can be obtained when the numbers of white and black faces are both odd. As discussed before, both correlations are then equal to zero.
Also observe that on the side of the six-vertex model the correlations involve observables that are local functions of the spins, whereas the loop observables 
depend on the global topology of all loops. Hence, slightly more care will be required when talking about convergence of the correlations under $\phi_{\n}$ as $\n\to \infty$, which is the subject of the next section.

Useful in this analysis will be the following interpretation of $\rho(\ell)$ for contractible loops.
For a contractible loop $\ell$, let $\delta(\ell)$ be the number of sources minus the number of sinks enclosed by the loop. Then
\begin{align} \label{eq:rho}
\rho(\ell) = \begin{cases}
1& \textnormal{ if $\delta(\ell) =0 \textnormal{ mod } 4$}, \\
-\tan \lambda &\textnormal{ if $\delta(\ell) =1 \textnormal{ mod } 4$}, \\
-1 &\textnormal{ if $\delta(\ell) =2 \textnormal{ mod } 4$}, \\
\tan \lambda &\textnormal{ if $\delta(\ell) =3 \textnormal{ mod } 4$}. 
\end{cases}
\end{align}
This can be obtained from~\eqref{def:rho} by computing the total flux of all the fixed directed paths through~$\ell$ (the number of times the paths cross $\ell$ from the inside to the outside minus 
the number of crossings from the outside to the inside). For topological reasons, this number is independent of the particular choice of paths, and is equal to~$\delta(\ell)$.
In particular if a contractible loop $\ell$ encloses all the sources and sinks, or none of them, then $\delta(\ell)=0$ and $\rho(\ell)=1$.
Moreover, if $\ell$ is noncontractible, then $\rho( \ell)=\epsilon(\vec \ell)$ if $\epsilon(\vec \ell)$ is real, and $\rho( \ell)=0$ otherwise. Here, $\vec \ell$ is any of the two orientations of $\ell$.
This means that for any $\Lambda \subset \mathbb T_\n$,
\begin{align} \label{eq:quasilocal}
\prod_{\ell\in L} \rho(\ell) \mathbf 1_{T_\Lambda}(L)=\prod_{\ell \in L \cap \Lambda} \rho(\ell)\mathbf 1_{T_\Lambda}(L),
\end{align} where, with a slight abuse of notation, $L\cap \Lambda$ are the loops from $L$ that are contained in $\Lambda$, and where $T_\Lambda$ is the set of loop configurations $L$ such that $L\cap \Lambda$ contains a contractible loop surrounding all the sources and sinks. 
Indeed, if $L\in T_\Lambda$, then for topological reasons,
\begin{itemize} 
\item any other contractible 
loop $\ell \in L$ that is not contained in $\Lambda$, either surrounds no sinks and no sources, or surrounds all of them, and hence by \eqref{eq:rho}, $\rho(\ell)=1$.
\item any noncontractible loop $\ell$ satisfies $\rho( \ell)=\epsilon(\vec \ell)=1$.
\end{itemize}
This is a form of \emph{locality} of the loop observables that we will later use to 
conclude convergence of their expectations in the infinite volume limit.

To finally make a connection with the random cluster model we follow Baxter, Kelland and Wu, and interpret the unoriented loops
as interfaces winding between clusters of open edges in a bond percolation configuration and its dual configuration
\[
\xi\in \Omega^{\bl}_\n=\{ 0,1\}^{E(\T^{\bl}_{\n})} \qquad \text{and} \qquad  \xi^{\dagger}\in \Omega^{\wi}_\n= \{ 0,1\}^{E(\T^{\wi}_{\n})}
\]
respectively (see~Fig.\ref{fig:BKW} and Fig.\ref{fig:large}). This yields a bijection between $\LL_{\n}$ and $\Omega^\bl_\n$, 
and we will write $L(\xi)$ for the unoriented loop configuration corresponding to $\xi$ under this map.
It turns out that the distribution of $\xi$ defined by formula~\eqref{eq:nu} is very closely related to the one of the \emph{critical random cluster model} which on the torus is given, up to normalizing constants, by 
\begin{align*} 
\phi_{\n}^{\textnormal{rc}}(\xi)  \propto q^{k(\xi)} \big(\tfrac{p_c}{1-p_c} \big)^{|\xi|} = q^{k(\xi)} \sqrt{q}^{|\xi|},
\end{align*}
where $p_c=\sqrt{q}/(1+\sqrt{q})$ is the critical parameter~\cite{BefDum}, and $k(\xi)$ is the number of connected component of $\xi$ (including isolated vertices) thought of as a subgraph of $\T^{\bl}_{\n}$.
Indeed, using Euler's formula for graphs drawn on the torus (see e.g.\ Lemma 3.9 in~\cite{Disc}), one can rewrite this as
\begin{align} \label{eq:nurc}
\phi_{\n}^{\textnormal{rc}}(\xi)  \propto \sqrt{q}^{|L(\xi)|}q^{s(\xi)}  ,
\end{align}
where
\[s(\xi)=
\begin{cases}
1 & \text{if $\xi$ is a net}, \\
0 & \text{otherwise}.
\end{cases}
\] 
Here, a \emph{net} is a subgraph of $\T_{\n}^{\bl}$ that contains two noncontractible cycles of different homotopy class.

\begin{figure}
\begin{center}
 \includegraphics[scale=0.7]{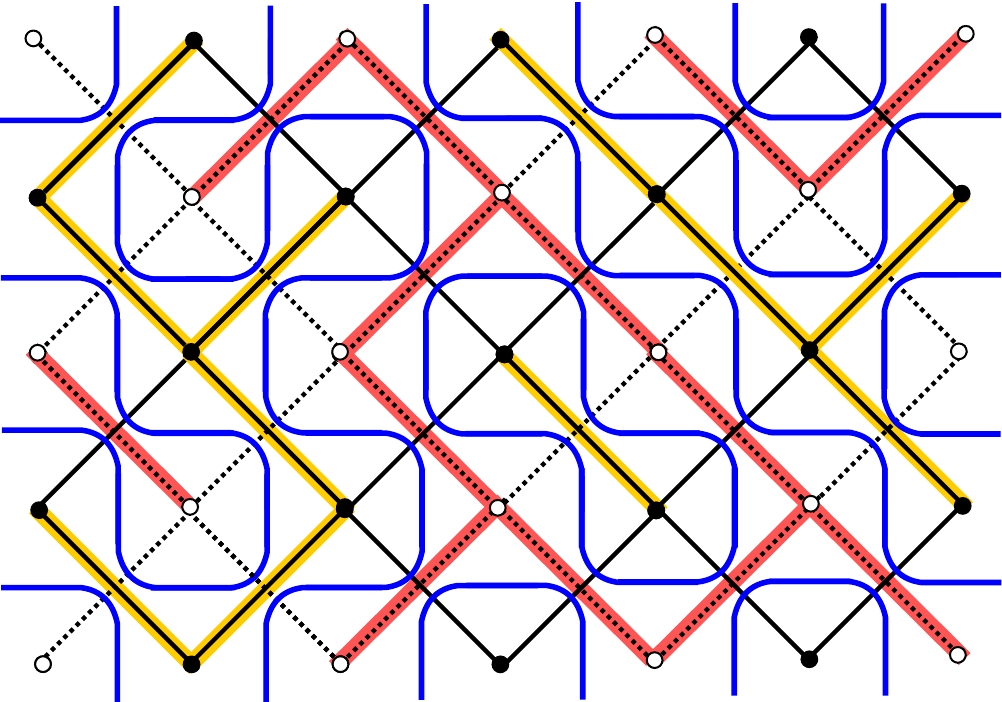}
\caption{
An unoriented loop configuration on a $6\times 4$ torus and the corresponding bond percolation configurations $\xi$ (yellow) and $\xi^{\dagger}$ (red).
The solid (resp.\ dotted) black lines represent the edges of $\T_{\n}^{\bl}$ (resp.\ $\T_{\n}^{\wi}$)
}
 \label{fig:large}
\end{center}
\end{figure}

\section{The infinite volume limit}
In this section we discuss convergence of finite volume measures as $\n\to \infty$ and as a result we prove Theorem~\ref{thm:existence}.
The main tools are the formulas for spin correlations from the previous section and the results on the critical random cluster model with $q\in[1,4]$ of Duminil-Copin, Sidoravicius and Tassion~\cite{DCST}.

\subsection{Convergence of $\phi_{\n}$}
Let $\Omega_\bl=\{0,1\}^{E(\mathbb Z_\bl^2)}$ and let $\mathcal F$ be the product $\sigma$-algebra on $\Omega_\bl$. 
Recall that $\Omega_\bl$ with the product discrete topology is a compact space for which $\mathcal F$ is the Borel $\sigma$-algebra. 

In what follows we think of $\Omega^\bl_\n$ as a subset of $\Omega_\bl$ by cutting the torus $\mathbb T_\n$ along the two noncontractible cycles at graph distance 
$nj_1$ in the horizontal and $nj_2$ in the vertical direction from the origin (as it was done in \eqref{eq:box}), and 
extending each percolation configuration~$\xi \in \Omega^\bl_\n$ to by setting its values to zero on the edges outside the resulting (rotated) box in $\mathbb Z^2_\bl$.
In particular, we think of $\phi_{\n}$ as a measure on~$(\Omega_\bl,\mathcal F)$.

The weak convergence of $\phi_{\n}$ to~$\phi$ will be a consequence of the fact that~$\phi$ is the unique critical random 
cluster measure on~$\mathbb Z_\bl^2$~\cite{DCST}, and the close relationship between formulas~\eqref{eq:nu} and~\eqref{eq:nurc}. 

\begin{lemma} \label{lem:phiconvergence}
For $c\in [\sqrt 3,2]$, $\phi_\n$ converges weakly to $\phi$ as $\n\to \infty$.
\end{lemma}

Before proving the result we recall some classical definitions. To this end, for $E\subset E(\mathbb Z_\bl^2)$, let $\mathcal F_E\subset \mathcal F$ be the $\sigma$-algebra generated by the 
states of the edges in~$E$. Also define $E^c = E(\mathbb Z_\bl^2) \setminus E$.
We say that a probability measure $\phi_0$ on~$\Space_\bl$ is \emph{insertion}
\emph{tolerant} if there exists $\epsilon>0$ such that for every $e\in E(\mathbb Z_\bl^2)$ and 
every event $A \in \mathcal F_{ \{e\}^c} $ of positive measure, we have
\begin{align} \label{eq:insertion}
\phi_0 (\xi(e) =1 \mid A) \geq \epsilon.
\end{align}
We say that $\phi_0$ is \emph{deletion tolerant} if the law of $1-\xi$ is insertion tolerant. Finally,~$\nu$ has \emph{finite energy} if it is both insertion and deletion tolerant.
A result that we will use is the classical Burton--Keane~\cite{BK} 
theorem saying that the probability of seeing more than one infinite cluster is zero under any 
translation invariant probability measure with finite energy.

We say that a probability measure $\phi_0$ on $(\Omega_\bl,\mathcal F)$ is a \emph{critical DLR random cluster measure} with parameter $q$ if for all $A\in \mathcal F$ and all 
finite boxes~$\Lambda\subset E(\mathbb Z^2_{\bl})$,
we have
\begin{align} \label{eq:DLR}
\phi_0(A \mid \mathcal F_{ \Lambda^c})(\zeta) = \phi_{\Lambda}^{\zeta}(A) \qquad \text{for $\phi_0$-a.e.\ $\zeta$},
\end{align}
where $\phi_{\Lambda}^{\zeta}$ is the \emph{critical random cluster measure with boundary conditions} $\zeta$ defined on
\[
\Omega_{\Lambda}^{\zeta}=\{ \xi \in \Omega: \xi(e)=\zeta(e) \textnormal{ for } e\in \Lambda^c\},
\]
and given by
\begin{align}  \label{eq:bcrc1}
\phi_{\Lambda}^{\zeta}(\xi)   \propto q^{k_{\Lambda}(\xi)} \sqrt{q}^{|\xi \cap \Lambda|} . 
\end{align}
Here $k_{\Lambda}(\xi)$ is the number of connected components of $\xi$ that intersect $\Lambda$. 
We note that if $\zeta$ contains at most one infinite cluster, then
\begin{align} \label{eq:bcrc}
\phi_{\Lambda}^{\zeta}(\xi)\propto \sqrt{q}^{|L_{\Lambda}(\xi)|},
\end{align}
where $L_{\Lambda}(\xi)$ are the loops (or biinfinite paths) in $L(\xi)$ that intersect $\Lambda$.
One can check this by establishing that both weights in~\eqref{eq:bcrc1} and \eqref{eq:bcrc} change in the same way after altering the state of a single edge.
The fundamental result for us will be that for $q\in [1,4]$, there exists exactly one critical DLR random cluster measure as was shown in~\cite{DCST}.

\begin{proof}[Proof of Lemma~\ref{lem:phiconvergence}]
Since $\Space_\bl$ is compact, the sequence $(\phi_\n)$ is tight and it is enough to prove that every subsequential limit $\phi_0=\lim_{k\to \infty} \phi_{\n_k}$ is equal to $\phi$. 
To show this, by the uniqueness result of~\cite{DCST}, we only need to check that $\phi_0$ satisfies the 
DLR condition \eqref{eq:DLR} for any box $\Lambda$.
We will do this by arguing that in the infinite volume limit, the value of ${|L_{\textnormal{nctr}}(\xi)|}$ in \eqref{eq:nu} does 
not depend on the state of $\xi$ inside $\Lambda$, which will imply that the conditional distribution of \eqref{eq:nu} simplifies to~\eqref{eq:bcrc}.

To be precise, note that by~\eqref{eq:nu} the measures $\phi_\n$ have finite energy with constants that are uniform in $\n$, and therefore  $\phi_0$ has finite energy as the weak limit of $\phi_{\n_{k}}$. Moreover, $\phi_0$ is clearly translation invariant. 
Therefore by the classical Burton--Keane argument~\cite{BK} the configuration
$\xi$ has at most one infinite connected component $\phi_0$-a.s. The same holds for $\xi^\dagger$ since it has the same distribution as~$\xi$ under $\phi_0$. For topological reasons, this means that $L(\xi)$ contains at most one infinite loop (by which we mean a biinfinite path)~$\phi_0$-a.s.\ which is the interface between these potential infinite primal and dual clusters.

Let $\Lambda_N\subset E(\mathbb Z^2_\bl)$ be a box of size $N\times N$ in $\mathbb Z^2_\bl$ centered at the origin.
For~$N$ such that $\Lambda\subseteq\Lambda_N$, let $s_N$ be the number of paths contained in $\Lambda_N\setminus \Lambda$ (that are parts of loops in $L(\xi)$) that intersect both $\Lambda$ and the outside of~$\Lambda_N$. Note that $s_N$ must be even, and let 
\[
S_{N}=\{\xi: s_N(\xi)\leq 2\}\in \mathcal F_{\Lambda^c} \cap \mathcal F_{\Lambda_N}.
\]
For future reference, also note that since there is at most one infinite loop $\phi_0$-a.s. and since $S_N$ is increasing in $N$, for every $B\in \mathcal F$, we have 
\begin{align} \label{eq:indicator}
\mathbf{1}_{B}=\lim_{N\to \infty}\mathbf{1}_{B\cap S_{N}}, \qquad \phi_0\textnormal{-a.s.}
\end{align}

We now fix $\varepsilon>0$ and take $k$ so large that the law of the percolation configuration
under $\phi_{\n_{k}}$ and $\phi_0$ restricted to $\Lambda_N$ are at total variation distance less than $\varepsilon$ from each other.
This is possible by the weak convergence of $\phi_{\n_{k}}$ to $\phi_0$ and since $N$ is fixed.
Now observe that for two configurations $\zeta, \zeta' \in S_{N}$ such that $\zeta=\zeta'$ on $\Lambda_N$, 
we have $\phi_{\Lambda}^{\zeta}=\phi_{\Lambda}^{\zeta'}$. 
This follows from~\eqref{eq:bcrc}, and the fact that if $s_N(\zeta)=2$, then necessarily the two paths crossing the annulus $\Lambda_N\setminus \Lambda$ 
must belong to the same loop in $L(\zeta)$ (the other case $s_N(\zeta)=0$ is clear).
Moreover, by \eqref{eq:nu} we have for $\zeta\in S_N$, 
\[
 \phi_{\n_k}(\cdot \mid \mathcal F_{ \Lambda^c})(\zeta)=  \phi_{\Lambda}^{\zeta}(\cdot)  \propto \sqrt{q}^{|L_{\Lambda}(\cdot)|}  
 \] since, on $S_{N}$, changing the state of 
an edge in~$\Lambda$ cannot change the number of noncontractible loops in~$L$ (by the same reasoning as above).
Hence, for all events $A\in \mathcal F_{\Lambda_N}$ and $B\in \mathcal F_{\Lambda^c}\cap \mathcal F_{\Lambda_N}$, we can write
\begin{align*}
\int_{B\cap S_{N}} \phi_{\Lambda}^{\zeta}(A)d\phi_0(\zeta)+O(\varepsilon) &=\int_{B\cap S_{N}} \phi_{\Lambda}^{\zeta}(A)d\phi_{\n_{k}}(\zeta) \\
&=\int_{B\cap S_{N}} \phi_{\n_k}(A \mid \mathcal F_{ \Lambda^c})(\zeta)d\phi_{\n_{k}}(\zeta) \\
&=  \phi_{\n_k}(A\cap B \cap S_N ).
\end{align*}
Taking first $k\to \infty$ and using the fact that $A\cap B \cap S_N $ is a local event, and then taking $N \to \infty$ and using \eqref{eq:indicator}, we get
\[
\int_B \phi_{\Lambda}^{\zeta}(A)d\phi_0(\zeta) = \phi_{0}(A\cap B )
\]
for all local events $A\in \mathcal F$ and $B\in \mathcal F_{\Lambda^c} $. This yields the DLR condition \eqref{eq:DLR} since the local events in
$ \mathcal F$ and $\mathcal F_{\Lambda^c}$ generate the respective $\sigma$-algebras.
\end{proof}

\subsection{Convergence of $\mu_{\n}$}
In this section we use the correlation identities from Lemma~\ref{lem:loopcorellations} to deduce weak convergence of $\mu_{\n}$ from the convergence of $\phi_{\n}$.

Recall the (local) map from spin configurations $\sigma$ to arrow configurations~$\alpha$.
Using this correspondence, from now on, we will think of $\mu_\n$ as a measure on $\Sigma_\n: = \{-1,1\}^{\mathbb {T}^{\bl}_\n} \times \{-i,i\}^{\mathbb {T}_\n^\wi}$.
Note that compared to the original definition, now~$\mu_\n$ also accounts for the independent coin flip that we used to decide the value of the spin on the fixed face~$u_0$.
Similarly to previous considerations, 
we will also think of $\Sigma_\n$ as a subset of $\Sigma := \{-1,1\}^{\mathbb {Z}^2_\bl} \times \{-i,i\}^{\mathbb {Z}^2_\wi}$ by setting the values of spins outside the box~\eqref{eq:box}
to $1$ or $i$ depending on the sublattice.
In particular, $\mu_\n$ becomes a measure on $(\Sigma,\mathcal G)$ where $\mathcal G$ is the product $\sigma$-algebra on~$\Sigma$.

We first show convergence of spin correlations.
\begin{lemma} \label{lem:loopcorrelationslimit} Let $u_1,\ldots, u_{2k} \in \mathbb Z_{\bl}^2$ and $v_1,\ldots, v_{2l} \in \mathbb Z_{\wi}^2$ 
be black and white faces of $\mathbb Z^2$ respectively, and let $\rho$ and $s$ be as in Lemma~\ref{lem:loopcorellations}. Then for $c\in [\sqrt 3,2]$, 
\begin{align} \label{eq:corfunctions}
\mathbf E_{\mu_{\n}} \Big[ \prod_{i=1}^{2k}\sigma(u_i) \prod_{j=1}^{2l}\sigma(v_j)  \Big] \to (-1)^{s}\mathbf E_{\phi} \Big[\prod_{\ell\in L} \rho(\ell)\Big] \quad \textnormal{as } \n \to \infty.
\end{align}
\begin{proof} 
We will use the locality property~\eqref{eq:quasilocal} and the fact that there are infinitely many loops in $L(\xi)$ 
surrounding all the faces $u_1,\ldots, u_{2k},v_1,\ldots, v_{2l}$ $\phi$-a.s.

To be precise, by Lemma~\ref{lem:loopcorellations} it is enough to show that
\begin{align} \label{eq:difference}
\Big| \mathbf E_{\phi_{\n}} \Big[\prod_{\ell\in L} \rho(\ell) \Big] - \mathbf E_{\phi} \Big[\prod_{\ell\in L} \rho(\ell)\Big]\Big|\to 0 \quad \textnormal{as } \n \to \infty.
\end{align}
To this end, recall that $\Lambda_N\subset E(\mathbb Z^2_\bl)$ is the box of size $N\times N$ in $\mathbb Z^2_\bl$ centered at the origin, and 
$L(\xi) \cap \Lambda_N$ is the set of loops in $L(\xi)$ that are contained in $\Lambda_N$.
Let $T_N\in\mathcal F_{\Lambda_N}$ be the event that there is a loop in $L(\xi) \cap \Lambda_N$ that surrounds all the faces $u_1,\ldots, u_{2k},v_1,\ldots, v_{2l}$. 
Note that $|\prod_{\ell\in L} \rho(\ell)| \leq C$ deterministically for some $C<\infty$ that depends only on the distances between the fixed faces.
From~\cite{DCST} we know that there are infinitely many loops that surround all the fixed faces $\phi$-a.s., and therefore $\phi(T_N)\to 1$ as $N\to \infty$. Hence, for $\varepsilon >0$ we can choose~$N$ so large that $\phi(T_N) > 1-\varepsilon/C$, and therefore
\begin{align*}
\Big| \mathbf E_{\phi} \Big[\prod_{\ell\in L} \rho(\ell)\Big] -  \mathbf E_{\phi} \Big[\prod_{\ell\in L} \rho(\ell) \mathbf{1}_{T_N}\Big]\Big|=
 \Big| \mathbf E_{\phi} \Big[\prod_{\ell\in L} \rho(\ell)\Big] - \mathbf E_{\phi} \Big[\prod_{\ell\in L \cap \Lambda_N} \rho(\ell) \mathbf{1}_{T_N}\Big]\Big| <\varepsilon,
\end{align*}
where we used the locality property~\eqref{eq:quasilocal} to obtain the equality.
Since the random variables $\prod_{\ell\in L \cap \Lambda_N} \rho(\ell)$ and $\mathbf{1}_{T_N}$ are local, by Lemma~\ref{lem:phiconvergence} we can now take~$M$ so large that
\[
\Big|\mathbf E_{\phi} \Big[\prod_{\ell\in L \cap \Lambda_N} \rho(\ell) \mathbf{1}_{T_N}\Big] - \mathbf E_{\phi_\n} \Big[\prod_{\ell\in L \cap \Lambda_N} \rho(\ell) \mathbf{1}_{T_N}\Big] \Big| <\varepsilon \quad \textnormal{and} \quad 
\phi_{\n}(T_N) > 1-2\varepsilon/C
\]
for all $\n\geq M$. Using~\eqref{eq:quasilocal} again, we altogether get an upper bound of $4\varepsilon$ on~\eqref{eq:difference} for $\n\geq M$.
\end{proof}
\end{lemma}

We are now able to prove the convergence part of Theorem~\ref{thm:existence}. 
\begin{proof}[Proof of part $(i)$ of Theorem~\ref{thm:existence}]
Note again that since~$\Sigma$ is compact, it is enough to prove that all subsequential limits of $\mu_\n$ are equal. By the lemma above, 
all these limits have the same correlation functions of the form \eqref{eq:corfunctions}. We finish the proof by noticing that the indicator function of any local 
event can be written as a linear combination of such correlation functions (see \eqref{eq:indA}). \end{proof}
We denote the limiting measure on $(\Sigma,\mathcal G)$ by $\mu$.

\begin{corollary} \label{cor:loopcorr}
 In the setting of Lemma~\ref{lem:loopcorrelationslimit}, we have 
\begin{align*}
\mathbf E_{\mu} \Big[ \prod_{i=1}^{2k}\sigma(u_i) \prod_{j=1}^{2l}\sigma(v_j)  \Big] = (-1)^{s}\mathbf E_{\phi} \Big[\prod_{\ell\in L} \rho(\ell)\Big].
\end{align*}
\end{corollary}

\subsection{Mixing property of $\mu$}
Let $\mathcal G_{\textnormal{even}} \subset \mathcal G$ be the $\sigma$-algebra of even events, i.e., events invariant under the global sign flip $\sigma \mapsto -\sigma$.
In this section we show that $\mu$ as a measure on $(\Sigma, \mathcal {G}_{\textnormal {even}})$ 
(and hence also as a measure on arrow configurations $\mathcal O$ equipped with the product $\sigma$-algebra) is mixing in the sense as in part $(ii)$ of Theorem~\ref{thm:existence}.
Our argument uses Lemma~\ref{lem:loopcorrelationslimit} and heavily relies on the mixing property of the random cluster measure $\phi$ established in~\cite{DCST}.

\begin{proof}[Proof of part $(ii)$ of Theorem~\ref{thm:existence}]
We present the proof in the language of spins. The corresponding statement for arrow configurations follows immediately.

Let $A,B\in \mathcal G_{\textnormal{even}}$ depend on the state of spins in finite square boxes $V,V'\subset \mathbb Z^2_\bl \cup \mathbb Z^2_\wi$ respectively. 
We have 
\begin{align}\label{eq:indA}
\mathbf 1_{A}(\sigma)&= \sum_{\tilde \sigma \in A} \prod_{v\in V}\tfrac12(1+\epsilon(v)\tilde\sigma(v)\sigma(v)) 
=  \tfrac1{2^{|V|}} \sum_{S\subseteq V}  \Big(\sum_{\tilde \sigma \in A} \prod_{v \in S} \epsilon(v)  \tilde \sigma(v) \Big) \prod_{v \in S}  \sigma(v),
\end{align}
where $\epsilon(v)=1$ if $v\in \mathbb Z^2_\bl$ and $\epsilon(v)=-1$ if $v\in \mathbb Z^2_\wi$. Since $A$ is invariant under sign change, only terms involving sets $S$ of even 
cardinality remain after the sum over~$\tilde \sigma$ is taken. 
This means that for some (explicit) coefficients $\beta_S$, $\beta'_{S'}$,
\begin{align*}
\mathbf 1_{A}(\sigma) = \mathop{\sum_{S \subseteq V}}_{|S| \textnormal{ even}} \beta_S \sigma(S),\quad \textnormal{ and } \quad \mathbf 1_{B}(\sigma) = \mathop{\sum_{S' \subseteq V'}}_{|S'| \textnormal{ even}} \beta'_{S' }\sigma(S'),
\end{align*}
where $\sigma(S)=\prod_{u\in S} \sigma(u)$, and therefore
\begin{align*}
\mu(A\cap B) -\mu(A)\mu(B) & = \mathop{\sum_{S\subseteq V,S' \subseteq V'}}_{|S|, |S'| \textnormal{ even}}\beta_{S}\beta_{S'}
( \mathbf E_{\mu} [ \sigma(S)  \sigma(S')] - \mathbf E_{\mu} [ \sigma(S)] \mathbf E_{\mu} [   \sigma(S')]) .
\end{align*}
Hence, to get \eqref{eq:mixing} it is enough to show that there exists $\kappa>0$ such that 
\begin{align} \label{eq:sigmafactorization}
| \mathbf E_{\mu} [ \sigma(S)  \sigma(S')] - \mathbf E_{\mu}[\sigma(S)] \mathbf E_{\mu}[\sigma(S')] | \leq Kd(V,V')^{-\kappa}  
\end{align}
for any pair of sets $S\subseteq V,S'\subseteq V'$ of even cardinality, where $K$ depends only on the size of $V$ and $V'$.

To this end, we use Lemma~\ref{lem:loopcorrelationslimit}, where we choose an equal number of sources and sinks in~$S$ (and hence also in $S'$), to get
\begin{align} \label{eq:rhom}
& \mathbf E_{\mu} [ \sigma(S)  \sigma(S')]  =(-1)^{s+ s'} \mathbf E_{\phi}\Big [ \prod_{\ell \in L} \rho_{S\cup S'}(\ell) \Big], \\
& \mathbf E_{\mu} [ \sigma(S)]  =(-1)^s \mathbf E_{\phi}\Big [ \prod_{\ell \in L} \rho_S(\ell) \Big] \nonumber \\
& \mathbf E_{\mu} [ \sigma(S')] = (-1)^{ s'}\mathbf E_{\phi}\Big [ \prod_{\ell \in L} \rho_{ S'}(\ell) \Big], \nonumber 
\end{align}
where $s$ and $ s'$ are the number of white sinks in $S$ and $ S'$, and where $\rho_{S\cup  S'}$ is defined for $S\cup  S'$
as in~\eqref{eq:rho}, and $\rho_S$ (resp.\ $\rho_{ S'}$) is defined for $S$ (resp.\ $ S'$) using the same (but properly restricted) choice of sinks and sources.
In particular, we have that $\rho_S= \rho_{S\cup  S'}$ and $\rho_{ S'}= \rho_{S\cup  S'}$ on loops not surrounding any face of $ S'$ and $S$ respectively. 
We note here that $S$ can contain an even or an odd number of, say, white faces. In the latter case, the two last correlations are equal to zero. However, as mentioned before, the
formula from Lemma~\ref{lem:loopcorrelationslimit} is still valid, and we chose to use it to have a uniform treatment of both cases.

Let $\Lambda_N ,\Lambda'_N \subset E( \mathbb Z^2_{\bl})$ be boxes of size $N\times N$ centered around the centers of $V$ and $V'$ respectively. 
Define $T_{N}\in\mathcal F_{\Lambda_N}$ to be the event that there is no loop in $L$ that intersects both $V$ and the complement of $\Lambda_N$. 
Analogously define $ T'_{N}\in\mathcal F_{ \Lambda'_N}$ for $ \Lambda'_N$ and $ V'$. 
By the strong RSW property of $\phi$ established in~\cite{DCST}, we know that there exists $\kappa'>0$ depending only on $q$, and $K_2<\infty$ depending on $q$ and the size of $V$ and $V'$, such that for all $N>0$,
\begin{align}\label{eq:RSW}
\phi(T_N\cap  T'_N)\geq 1-K_2 N^{-\kappa'}.
\end{align}
Indeed, to ensure $T_N$, it is enough to construct an open circuit in the percolation configuration $\xi$ that surrounds $V$ and stays within $\Lambda_N$. 
By the strong RSW property and the positive association of $\phi$, this can be done with constant probability
for every annulus in a properly defined sequence of disjoint concentric and exponentially growing annuli centered around $V$. We leave the details of this standard argument to the reader.

Moreover we have 
\[
\big |\prod_{\ell \in L} \rho_{U}(\ell) \big| \leq \max(\tan \lambda, 1)^{ |V| +|V'|}  = : K_1
\]
deterministically for $U= S, S', S\cup  S' $, since there can be at most $|V| +|V'|$ loops intersecting $V\cup V'$. 
Hence, by \eqref{eq:rhom}, \eqref{eq:RSW} we have
\begin{align} \label{eq:in1}
& \Big|\mathbf E_{\mu} [ \sigma(S)  \sigma( S')]-\mathbf E_{\phi} \Big[ \prod_{\ell \in L} \rho_{S\cup  S'}(\ell)\mathbf 1_{T_N}\mathbf 1_{ T'_N}\Big]\Big| \leq K_3N^{-\kappa'},  \\ 
 &\Big| \mathbf E_{\mu} [ \sigma(S)] -\mathbf E_{\phi}\Big [ \prod_{\ell \in L\cap \Lambda_N} \rho_S(\ell)\mathbf 1_{T_N}\Big]\Big|   \leq K_3N^{-\kappa'}, \nonumber\\
  &\Big| \mathbf E_{\mu} [ \sigma(S')] -\mathbf E_{\phi}\Big [ \prod_{\ell \in L\cap  \Lambda'_N} \rho_{ S'}(\ell)\mathbf 1_{ T'_N}\Big]\Big|   \leq K_3N^{-\kappa'}, \nonumber
\end{align}
where $K_3=K_1K_2$.
Combined with the fact that the spin correlations are by definition bounded by one, the last two inequalities give
\begin{align}  \label{eq:in3}
&\Big| \mathbf E_{\mu} [ \sigma(S)]  \mathbf E_{\mu} [ \sigma(S')] -\mathbf E_{\phi}\Big [ \prod_{\ell \in L\cap \Lambda_N} \rho_S(\ell)\mathbf 1_{T_N}\Big]
 \mathbf E_{\phi}\Big [ \prod_{\ell \in L\cap  \Lambda'_N} \rho_{ S'}(\ell)\mathbf 1_{ T'_N}\Big]\Big| \\ 
 & \qquad \qquad \leq  K_3N^{-\kappa'} (K_3N^{-\kappa'}+2) .  \nonumber
\end{align}

On the other hand, we have
\begin{align} \label{eq:in2}
 \prod_{\ell \in L}  \rho_{S\cup \tilde S}(\ell)\mathbf 1_{T_N}\mathbf 1_{ T'_N}=  \prod_{\ell \in L \cap \Lambda_N} 
 \rho_S(\ell)\mathbf 1_{T_N} \prod_{\ell \in L\cap \Lambda'_N}\hspace{-0.3cm} \rho_{ S'}(\ell) \mathbf 1_{ T'_N}
\end{align}
whenever $\Lambda_N$ and $ \Lambda'_N$ are disjoint.
Moreover, since these two factors are local functions depending only on the state of edges in $\Lambda_N$ and $ \Lambda'_N$ 
respectively, by the mixing property of the critical random cluster model from in Theorem~5 of~\cite{DCST}, we have
\begin{align} \label{eq:in4}
&\Big|\mathbf E_{\phi} \Big[ \prod_{\ell \in L \cap \Lambda_N} 
 \rho_S(\ell)\mathbf 1_{T_N} \prod_{\ell \in L\cap \Lambda'_N}\hspace{-0.3cm} \rho_{ S'}(\ell) \mathbf 1_{ T'_N} \Big]  - 
\mathbf E_{\phi} \Big[ \prod_{\ell \in L \cap \Lambda_N}  \rho_S(\ell)\mathbf 1_{T_N}\Big]\mathbf E_{\phi} \Big[  \prod_{\ell \in L\cap \Lambda'_N}\hspace{-0.3cm} \rho_{ S'}(\ell) \mathbf 1_{ T'_N}\Big] \Big|
\\ & \nonumber\qquad \leq\mathbf E_{\phi} \Big[ \prod_{\ell \in L \cap \Lambda_N}  |\rho_S(\ell)|\mathbf 1_{T_N}\Big]\mathbf E_{\phi} \Big[  \prod_{\ell \in L\cap \Lambda'_N}\hspace{-0.3cm}| \rho_{ S'}(\ell)| \mathbf 1_{ T'_N}\Big] 
\big(\tfrac{N}{ d(\Lambda_N,\Lambda'_N)+N} \big)^{\kappa''} 
\\&\nonumber \qquad \leq  K_1^2 \big(\tfrac{N}{ d(\Lambda_N,\Lambda'_N)+N} \big)^{\kappa''} 
\end{align}
whenever $d(\Lambda_N,\Lambda'_N)\geq N$ for some $\kappa''>0$ that depends only on $q$. 
Combining this with \eqref{eq:in1}, \eqref{eq:in3}, \eqref{eq:in2} and \eqref{eq:in4} we obtain that the left-hand side of \eqref{eq:sigmafactorization} is at most 
\[
K_3N^{-\kappa'} (K_3N^{-\kappa'}+3) + K_1^2\big(\tfrac{N}{ d(\Lambda_N,\Lambda'_N)+N} \big)^{\kappa''} \leq K_4\big(N^{-\kappa'} +\big(\tfrac{N}{ d(\Lambda_N,\Lambda'_N)+N} \big)^{\kappa''}\big)
\] 
for $d(\Lambda_N,\Lambda'_N)\geq N$, where $K_4$ depends only on $q$ and the size of $V$ and $V'$.
Taking $N=\lfloor \sqrt{d(V,V')}\rfloor$ we show \eqref{eq:sigmafactorization} and complete the proof.
\end{proof}

We note that ergodicity of $\mu$ follows by using standard arguments where one approximates translation invariant events by local events, and then uses the established mixing property.

\subsection{Decorrelation of monochromatic spins} In this section we study the decay of spin correlations for spins on faces of the same color, and without loss of generality we choose the black faces.
The simplest case of Corollary~\ref{cor:loopcorr} says that for $u,u'\in \mathbb Z^2_{\bl}$,
\begin{align} \label{eq:twopoint}
\mathbf E_{\mu}[\sigma(u)\sigma(u')] = \mathbf E_{\phi} \Big[\prod_{\ell\in L} \rho(\ell)\Big] =  \mathbf E_{\phi} \Big[\rho^{N(u,u')}(-\rho)^{N(u',u)}\Big] ,
\end{align}
where $N(u,u')$ is the number of loops in $L=L(\xi)$ which surround $u$ but not~$u'$.
Note that for $c=2$, we have $\rho=0$ and the right-hand side becomes $\mathbf E_{\phi} [N(u,u')=0]$ (see Remark~\ref{rem:c2}).

Using that $\rho = \tanh \lambda \leq 1$ for $c\in [\sqrt{2+\sqrt 2},2]$ we obtain the following result.

\begin{theorem} \label{thm:decorrelation}
 For $c\in [\sqrt{2+\sqrt 2},2]$, there exists $\theta=\theta(c)>0$ such that for all $u,u'\in \mathbb Z^2_{\bl}$,
\begin{align} \label{eq:decorrelation1}
 \mathbf E_{\mu}[\sigma(u)\sigma(u')]\leq |u-u'|^{-\theta}.
\end{align}
\end{theorem}
We note that positivity of this two-point function follows from the percolation representation of the spin model described in Section~\ref{sec:omega}. 
We also note that our argument does not give much information on the value of the exponent $\theta$.

\begin{proof}
We consider two cases.

\textbf{Case I:} $c\in (\sqrt{2+\sqrt 2},2]$. In this case $\rho<1$, and we can simply bound the right-hand side of \eqref{eq:twopoint} from above by $\mathbf E_{\phi} [\rho^{N(u,u')}] $. 
Note that $N(u,u')$ is bounded from below by the number of loops surrounding $u$ whose diameter is smaller that $|u-u'|$. This number on the other hand stochastically dominates a binomial random variable with $\log |u-u'|$ trials and with (uniformly in $u,u'$) positive success probability.
This is a consequence of the strong RSW results for the random cluster model obtained in~\cite{DCST}. Indeed, using the positive association of the measure and the 
the fact that one can cross long rectangles with uniform positive probability and under arbitrary boundary conditions, one can iteratively construct circuits of $\xi$ and $\xi^{\dagger}$
in exponentially growing annuli around $u$. Each pair of such consecutive clusters of $\xi$ and $\xi^{\dagger}$ contributes one loop to $L(\xi)$ that surrounds $u$ but not $u'$.
This yields \eqref{eq:decorrelation1} by using elementary properties of binomial distribution. We leave the details to the reader.

\textbf{Case II:} $c= \sqrt{2+\sqrt 2}$. In this case $\rho=1$ and the right-hand side of \eqref{eq:twopoint} simplifies to $\mathbf E_{\phi} [(-1)^{N(u',u)}] $.
Let $v=u+(1,0),v'=u'+(1,0)$ be the two vertices of $\mathbb Z^2_{\wi}$ directly to the right of $u$ and $u'$, and let $e,e'\in E(\mathbb Z^2)$ 
be the edges separating $u$ from $v$, and $u'$ from $v'$ respectively. Since the law of $L(\xi)$ is invariant under translation by $(1,0)$, we 
have that $N(u',u)$ has the same distribution as $N(v',v)$, and we can write  
\begin{align} \label{eq:cancellations}
\mathbf E_{\mu}[\sigma(u)\sigma(u')] & =\tfrac 12 \mathbf E_{\phi} [(-1)^{N(u',u)}+(-1)^{N(v',v)}].
\end{align}
We now notice that for each configuration of $\xi$, we have $N(u',u)(\xi)=N(v',v)(\xi)$ if there is a loop in $L(\xi)$ which goes through 
$e$ and surrounds $u'$ or $v'$, or there is a loop that goes through $e'$ and surrounds $u$ or $v$. 
Moreover, in this case $N(u',u)$ is even. 
Otherwise we have $N(u',u)(\xi)=N(v',v)(\xi)\pm 1$ and the corresponding two terms in the expression 
above cancel out. All in all we obtain that \eqref{eq:cancellations} is bounded above by the probability that the cluster in $\xi$ of either $u$, $u'$, $v$ or $v'$ has radius larger than $|u-v|$.
Again by the RSW property of the critical random cluster measure, this probability decays polynomially in $|u-v|$, and we finish the proof.
\end{proof}

\begin{remark}
We want to stress the fact that such polynomial decorrelation (including a polynomial lower bound) for monochromatic spins is expected to hold for all positive $c$. However, so far we were not able to obtain it 
using \eqref{eq:decorrelation1}. The reason is that in the case when $\rho>1$ one needs to argue that the fluctuations of the 
random sign $(-1)^{N(u',u)}$ and the exponential growth of $\rho^{N(u,u')+N(u',u)}$ cancel out to order $O(|u-u'|^{-\theta})$. Note that \eqref{eq:decorrelation1} already implies 
(since the left-hand side is bounded above by one) that such 
cancellations occur to order~$O(1)$.
\end{remark}

\begin{remark}
By arguments as in the previous section, polynomial decorrelation of monochromatic spins yields a similar mixing property of $\mu$ for all local events (not only even local events).
\end{remark}

\section{Delocalization of the height function}\label{sec:details}
In this section we combine the framework developed in~\cite{Lis19} with the results from the previous sections to prove delocalization of the height function for $c\in[\sqrt{2+\sqrt 2},2]$.
To this end, we need to consider a conditioned version of the six-vertex model.
We define $\mathcal O^0_{\n}\subset \mathcal {O}_\n$ to be the set of arrow configurations such that the spin system $\sigma$ is globally well defined on $\T_{\n}$. In other words, 
these are the arrow configurations such that the increment of the height function along any noncontractible cycle in the dual graph $\mathbb T_\n^*$ is zero mod $4$.
We denote by $\mu^0_{\n}$ the measure $\mu_{\n}$ conditioned on $\mathcal O^0_{\n}$. As before, we will identify $\mu^0_{\n}$ with a probability measure on the set of spin configurations $\Sigma_\n$, which we think of as a subset of~$\Sigma$.

\subsection{Convergence of $\mu^0_{\n}$}
We will first show that $\mu^0_{\n}$ also converges to $\mu$ as $\n\to \infty$.
The argument is analogous to the one used to establish convergence of $\mu_\n$ itself, and we will only focus here on the (topological) differences arising from the conditioning on $\mathcal O^0_{\n}$.

To this end, we perform the same steps as in the unconditional BKW representation.
We first expand the arrow configurations in $\mathcal O^0_{\n}$ to obtain a set of fully-packed oriented loop configurations, denoted by $\vec {\mathcal {L}}^0_\n$. 
We denote the sets of oriented and unoriented loop configurations composed of only contractible loops by $\vec {\mathcal {L}}^{\textnormal{ctr}}_\n$ and $ {\mathcal {L}}^{\textnormal{ctr}}_\n$ respectively.
We now notice that any contractible oriented loop contributes zero to the increment of the height function along any noncontractible cycle. 
Hence, $\vec {\mathcal {L}}^{\textnormal{ctr}}_\n \subset \vec {\mathcal {L}}^0_\n$ and the complex measure induced on $\vec {\mathcal {L}}^{\textnormal{ctr}}_\n$ 
and the probability measure induced on $ {\mathcal {L}}^{\textnormal{ctr}}_\n$ by $\mu^0_\n$ is the same as that induced by~$\mu_{\n}$. 

To treat the case involving noncontractible loops, we recall a topological fact saying that for a simple noncontractible closed curve on the torus, the algebraic numbers $(k,l)$
of times the curve intersects the equator and a fixed meridian respectively are coprime (in particular, one of them has to be odd). 
Moreover, such pairs of numbers $(k,l)$ are in a one-to-one correspondence with isotopy classes of such curves.
Let $\vec L \in \vec {\mathcal {L}}_\n \setminus \vec {\mathcal {L}}^{\textnormal{ctr}}_\n$ contain noncontractible loops.
Note that since these loops do not intersect, they have to be, up to orientation, of the same isotopy class $(k,l)$. 
Moreover, since the torus $\mathbb T_\n$ is of even size, the total increment of the height function must be even along any noncontractible loop. 
Combined with the fact that at least one of the numbers $(k,l)$, say $k$, is odd, this means that there must be an even number, say $2m$, of noncontractible loops in $\vec L$.
Let $m_1$ and $m_2$ be the numbers of such loops which intersect the meridian from right to left and from left to right respectively. In particular $m_1+m_2=2m$.
Then the increment of the height function of $\vec L$ is $a=(m_1-m_2)k$ along the meridian and $b=\pm(m_1-m_2)l$ along the equator.
If we now reverse the orientation of the noncontractible loop $\vec \ell_0\in \vec L$ 
which goes through the vertex with the smallest number (in some fixed ordering), we obtain a configuration $\vec L'$ for which these increments are
$a'=(m_1-m_2\pm2)k$ and $b'=\pm(m_1-m_2\pm2)l$ respectively. Since $m_1-m_2$ is even, we have the following cases:
if $l$ is even, then $a=a'+2 \textnormal{ (mod 4)}$ and $b=b'=0 \textnormal{ (mod 4)}$, and if $l$ is odd, then $a=b=a'+2 =b'+2\textnormal{ (mod 4)}$. 
In both situations, exactly one of the two configurations $\vec L$ and $\vec L'$ belongs to $\vec {\mathcal {L}}^0_\n$. Note that the correspondence~$\vec L \leftrightarrow \vec L'$ 
is involutive and measure preserving (since $\vec \ell_0$ has total winding zero). This implies that exactly half (in terms of the induced complex measure) oriented 
loop configurations in $\vec {\mathcal {L}}_\n \setminus \vec {\mathcal {L}}^{\textnormal{ctr}}_\n$
belong to $\vec {\mathcal {L}}^0_\n$.
In particular we get the following formula for the induced probability measure on $\mathcal{L}_\n$, 
\begin{align} \label{eq:nu0}
\phi^0_\n( L) =  \frac{1}{Z^0_{\n}}  \sqrt{q}^{|L|} \big(\tfrac{2}{\sqrt q}\big)^{|L_{\textnormal{nctr}}|}\tfrac12 \big(1+\mathbf 1_{ {\mathcal {L}}^{\textnormal{ctr}}_\n}(L)\big),
\end{align}
where $L_{\textnormal{nctr}}$ is the set of noncontractible loops in $L$.

As a result of considerations exactly like in the previous section, we obtain the following convergence.
\begin{proposition} \label{lem:0conv}
For $c\in [\sqrt 3, 2]$, $\mu^0_\n \to \mu$ weakly as $\n \to \infty$.
\end{proposition}

\subsection{The percolation process $\omega$} \label{sec:omega}

We follow~\cite{GlaPel,Lis19} and define a bond percolation model $\omega$ on top of the spin configuration $\sigma$ sampled according to $\mu^0_\n$ 
(the corresponding parameters in~\cite{Lis19} are $q=q'=2$ and $a=b=c^{-1}$). 
We note that the model was also used in~\cite{SpiRay} in the study of the six-vertex model in the localized regime.

Recall that the graphs $\mathbb{T}^{\bl}_\n$ and $\mathbb{T}^{\wi}_\n$ (likewise $\mathbb {Z}^2_\bl$ and $\mathbb {Z}^2_\wi$) are dual to each other, 
and denote by $\sigma^{\bl}$ and $\sigma^{\wi}$ the restrictions 
of the spin configuration $\sigma$ to the vertices of the respective graphs. 
We now define $\eta(\sigma^{\wi}) \subseteq E(\mathbb{T}^{\bl}_\n)$ to be the set of \emph{contours} of $\sigma^{\wi}$, i.e., edges  
whose {dual} edge in $E(\mathbb{T}^{\wi}_\n)$ carries two different values of the spin~$\sigma^{\wi}$ at its endpoints.
Given $\sigma$, to obtain the percolation configuration $\omega  \subseteq E(\mathbb{T}^{\bl}_\n)$, we proceed in steps:
\begin{itemize}
\item[$(i)$] we start with the configuration where all edges are closed,
\item[$(ii)$] we then declare each edge in $\eta(\sigma^{\wi})$ open,
\item[$(iii)$] for each edge $\{u,u'\} \in   E(\mathbb{T}^{\bl}_\n)$ 
still closed after step $(ii)$ and such that $\sigma(u) = \sigma(u')$, we toss an independent coin with success probability $1-1/c$. 
On success, we declare the edge open, and otherwise we keep it closed,
\item[$(iv)$] we denote by $\omega$ the set of all open edges.
\end{itemize}
Note that in particular $\eta(\sigma^{\wi})\subseteq \omega$.
We will write $\mathbf P_\n$ for the probability measure on configurations $(\sigma,\omega)\in \Sigma_\n \times \Omega_\n^{\bl}$ obtained from 
these steps when $\sigma$ is distributed according to $\mu^0_\n$.
Since the above procedure is local, independent for different edges, and invariant under the global sign change $\sigma \mapsto -\sigma$, from Proposition~\ref{lem:0conv} we immediately conclude the following.
\begin{corollary} \label{cor:P}
 $\mathbf P_\n$ converges weakly as $\n\to \infty$ to a probability measure $\mathbf P$ on 
$(\Sigma\times \Omega_{\bl}, \mathcal G \otimes \mathcal F)$ which is translation invariant, satisfies a mixing property as in Theorem~\ref{thm:existence}, and hence is ergodic on $\mathcal G_{\textnormal{even}} \otimes \mathcal F$ with respect to the translations of~$\mathbb Z^2_{\bl}$.
\end{corollary}

The following result connecting the percolation properties of $\omega$ under $\mathbf P$ with the behaviour of the height function under $\mu$ was proved in~\cite{Lis19}.
\begin{lemma} \label{lem:Lis19}
For $c\in [\sqrt{3},2]$, if 
\[
\mathbf P( \exists \textnormal{ an infinite cluster of }\omega) = 0,
\] 
then 
\begin{align*}
\mathbf{Var}_{\mu} [h(u)] \to \infty \quad \text{ as } \quad |u| \to \infty,
\end{align*}
where $u \in \mathbb Z^2_{\wi} \cup \mathbb Z^2_{\bl} $ is a face of $\mathbb Z^2$. 
\end{lemma}

Therefore, to prove Theorem~\ref{thm:delocalization} it is enough to show the following.
\begin{proposition}\label{prop:nopercolation}For $c\in [\sqrt{2+\sqrt 2},2]$, $\mathbf P( \exists \textnormal{ an infinite cluster of }\omega) = 0$.
\end{proposition}

We devote the rest of this section to the proof of this result. We note that percolation properties of related models were studied in~\cite{HolLi}.
We first recall a crucial property of the coupling between $\sigma$ and $\omega$ given by the following description of the conditional law of $\sigma^{\bl}$ given $\omega$~\cite{GlaPel,SpiRay,Lis19}, which is directly analogous to the Edwards--Sokal coupling between the Potts model and the random cluster model~\cite{EdwSok}.

\begin{lemma}[Edwards--Sokal property of $\omega$ and $\sigma^{\bl}$]\label{lem:EdSo}
Under the probability measure $\mathbf P_\n$, conditionally on~$\omega$, the spins $\sigma^\bl$ are distributed like an independent uniform assignment of a $\pm 1$ spin 
to each connected component of $\omega$. The same is true for $\mathbf P$ given that $\mathbf P( \exists \textnormal{ an infinite cluster of }\omega)=0$.
\end{lemma}
As a direct consequence we obtain a relation between connectivities in $\omega$ and spin correlations,
\begin{align} \label{eq:EScor}
\mathbf P_{\n}(u \textnormal{ connected to } u' \textnormal{ in } \omega)  = \mathbf{E}_{\mu^0_{\n}}[\sigma(u)\sigma(u')].
\end{align}
The idea now is to use this identity and the decorrelation of spins from Theorem~\ref{thm:decorrelation} to conclude no percolation for $\omega$.
\begin{remark} \label{rem:c2}
For $c=2$, both the distribution of $\omega$ under $\mathbf P$ and of $\xi$ under $\phi$ are the critical random cluster model with $q=4$ (see \cite{Lis19}).
In this case we know that formula \eqref{eq:EScor} also holds in the infinite volume (since $\omega$ does not percolate), and it is identical to formula \eqref{eq:example}
since for $q=4$, we have $\rho=0$ and the event that there is no loop separating $u$ from $u'$ in $L(\xi)$ is the same as the event of $u$ being connected to $u'$ in~$\xi$.
\end{remark}

We will first need to prove that there is at most one infinite cluster in $\omega$ under~$\mathbf P$. 
To this end,
we start with establishing insertion tolerance of $\omega$.
\begin{lemma}[Insertion tolerance of $\omega$] \label{lem:insertion}
For $c\in [\sqrt{3},2]$, the law of $\omega$ under $\mathbf P$ is insertion tolerant as defined in~\eqref{eq:insertion}.
\end{lemma}

\begin{proof}
Since $\mathbf P$ is the weak limit of $\mathbf P_\n$, it is enough to prove that $\mathbf P_\n$ satisfies \eqref{eq:insertion} with a constant $\epsilon>0$ that is independent of $\n$.

To this end, for a configuration $\zeta \in \Omega^{\bl}_\n $ and an edge $e=\{u_1,u_2\} \in E(\mathbb T^{\bl}_\n)$, let $\zeta^e,\zeta_e \in \Omega^{\bl}_\n$ be the configurations that 
agree with $\zeta$ on $E(\mathbb T^{\bl}_\n)\setminus \{ e\}$ and such that $\zeta^e(e)=1$ and $\zeta_e(e)=0$.
Note that by Lemma~\ref{lem:EdSo} we have
\begin{align}
\mathbf P_\n (\omega = \zeta^e)&\geq\mathbf P_\n (\omega = \zeta^e, \sigma(u_1) =\sigma(u_2),  \sigma(v_1) =\sigma(v_2)) \nonumber  \\
& =\tfrac p{1-p} \mathbf P_\n (\omega = \zeta_e, \sigma(u_1) =\sigma(u_2),  \sigma(v_1) =\sigma(v_2)) \nonumber \\
& = \tfrac p{1-p} \mathbf P_\n (\omega = \zeta_e, \sigma(u_1) =\sigma(u_2)) \nonumber\\
& \geq \tfrac 12 \tfrac p{1-p}  \mathbf P_\n (\omega = \zeta_e) \nonumber \\
& = \tfrac 12(c-1) \mathbf P_\n (\omega = \zeta_e) \label{eq:insertionbound} ,
\end{align}
where $\{v_1,v_2\}\in E(\mathbb{T}^{\wi}_\n)$ is the dual edge of $e$, and $p=1-c^{-1}$ is the success probability from step $(iii)$ of the definition of $\omega$.
To get the first equality, we used the fact that if $\sigma(u_1) =\sigma(u_2)$ and $\sigma(v_1) =\sigma(v_2)$, then we can open $e$ only in step $(iii)$ by tossing a coin. 
In the second equality, we used that if $e$ is closed then necessarily $\sigma(v_1) =\sigma(v_2)$. The last inequality follows from Proposition~\ref{lem:EdSo} and the fact that,
on the event that $v_1$ is not connected to $v_2$ in $\omega$, both faces obtain independent $\pm 1$ spins (otherwise, they must have the same spin).
From \eqref{eq:insertionbound} we get that
\[
\mathbf P (\omega(e)=1 \mid \omega(e') = \zeta(e') \textnormal{ for } e'\neq e) \geq \tfrac{c-1}{c+1} ,
\]
and hence \eqref{eq:insertion} holds true with $\epsilon = (c-1)/(c+1)$. This ends the proof.
\end{proof}

We will now exclude the possibility of more than one infinite clusters in $\omega$ under~$\mathbf P$. Since the law of $\omega$ is
not deletion tolerant in the sense of \eqref{eq:insertion}, we need to slightly modify the classical argument of Burton and Keane~\cite{BK}.
(Actually, one can always remove an edge from $\omega \setminus \eta(\sigma^{\wi})$ by paying a constant price, but removing edges from $ \eta(\sigma^{\wi})$
cannot be done locally and the cost can be arbitrarily high).
\begin{lemma} \label{lem:BK}
For $c\in [\sqrt{3},2]$, 
\begin{align*} 
\mathbf P(\exists \textnormal{ more than one infinite cluster of $\omega$})=0.
\end{align*}
\end{lemma}
\begin{proof}[Proof of Lemma~\ref{lem:BK}]
By Corollary~\ref{cor:P}, $\mathbf P$ is translation invariant and ergodic when projected to $\mathcal F$, and by Lemma~\ref{lem:insertion}, it is insertion tolerant. Hence, by classical arguments we have 
that
\[
\mathbf P(\exists \textnormal{ more than one but finitely many infinite clusters of $\omega$})=0.
\]
To conclude the proof we therefore need to show that
\begin{align} \label{eq:Cinfty}
\mathbf P(\exists \textnormal{ infinitely many infinite clusters of $\omega$})=0.
\end{align}
To this end, we say that $0\in \mathbb Z^2_{\bl}$ is a \emph{trifurcation} if it belongs to an infinite cluster of $\omega$
that splits into exactly three infinite and no finite clusters after removing $0$ and the edges incident on $0$.
We assume by contradiction that the probability in \eqref{eq:Cinfty} is equal to $1$ (we can assume this by ergodicity) of $\mathbf P$. We will show that under this assumption 
\[
\mathbf P(0 \textnormal{ is a trifurcation})>0.
\]
This will yield the desired contradiction in the same way as in the original argument of Burton and Keane. 

In what follows, we will construct trifurcations by modifying (in steps) the configuration $(\sigma,\omega)$ inside a large but finite box. 
To this end, for $\Lambda\subset E(\mathbb Z^2_\bl)$, let
\[
C_6(\Lambda) = \{ \partial \Lambda \textnormal{ intersects at least \emph{six} infinite clusters of } \omega|_{  \Lambda^c} \} \in\mathcal G \otimes \mathcal F_{\Lambda^c},
\]
where $\partial \Lambda$ is the set of vertices of $\Lambda$ adjacent to a vertex outside $\Lambda$, and $\omega|_{  \Lambda^c}$ is the restriction of the configuration $\omega$ to the edges of $\Lambda^c$. 
We now fix $\Lambda$ to be a square box large enough so
that for $C_6=C_6(\Lambda)$, 
\begin{align*} \label{eq:I6}
\mathbf P(C_6 )>1/2.
\end{align*}
For a set of black vertices $B$ and white vertices $W$, we define
\begin{align*}
S_{\pm}(B) &= \{ \sigma  \textnormal{ is constant and equal } \pm 1 \textnormal{ on } B\} \in \mathcal G \otimes \mathcal F, \textnormal{ and} \\
 S_{\pm}(W)& = \{  \sigma  \textnormal{ is constant and equal } \pm i \textnormal{ on } W\} \in \mathcal G \otimes \mathcal F.
\end{align*}
Note that by the Edwards--Sokal property from Lemma~\ref{lem:EdSo}, for any event $I\in \mathcal G \otimes \mathcal F$ depending only on $\omega$, we have  
\[
\mathbf P_{\n} (S_+(B) \mid I ) \geq (\tfrac 1 2)^{|B|}
\]
independently of $\n$. Hence, 
by the weak convergence of $\mathbf P_{\n}$ to $\mathbf P$ we know that
\[
\mathbf P (S_+(V(\Lambda)) \mid C_6 ) \geq (\tfrac 12) ^{|V( \Lambda)|}.
\]

Recall that $\eta(\sigma^\wi)\subset E(\mathbb Z^2_\bl)$ is the set of interfaces separating spins of different value in $\sigma^\wi$ which in turn is the restriction of $\sigma$ to $\mathbb Z^2_\wi$.
The crucial observation now is that for each $(\sigma,\omega) \in  S_+(V(\Lambda))\cap C_6$, one can choose a constant sign $\varsigma=\varsigma(\sigma,\omega)=\pm 1$ such that
there are at least \emph{three} infinite clusters in $ \omega|_{ \Lambda^c} \cup \eta( \sigma_{\varsigma}^\wi)$, where 
\[ \sigma_{\varsigma}(u)= \begin{cases} \varsigma & \textnormal{ for } u\in V(\Lambda^*), \\
\sigma(u)  & \textnormal{ otherwise},
\end{cases}
\]
and where $\Lambda^*\subset E(\mathbb Z^2_{\wi})$ is the box whose vertices are the bounded faces of $\Lambda$ (see Fig.~\ref{fig:BK}). 
\begin{figure}
\begin{center}
 \includegraphics[scale=1]{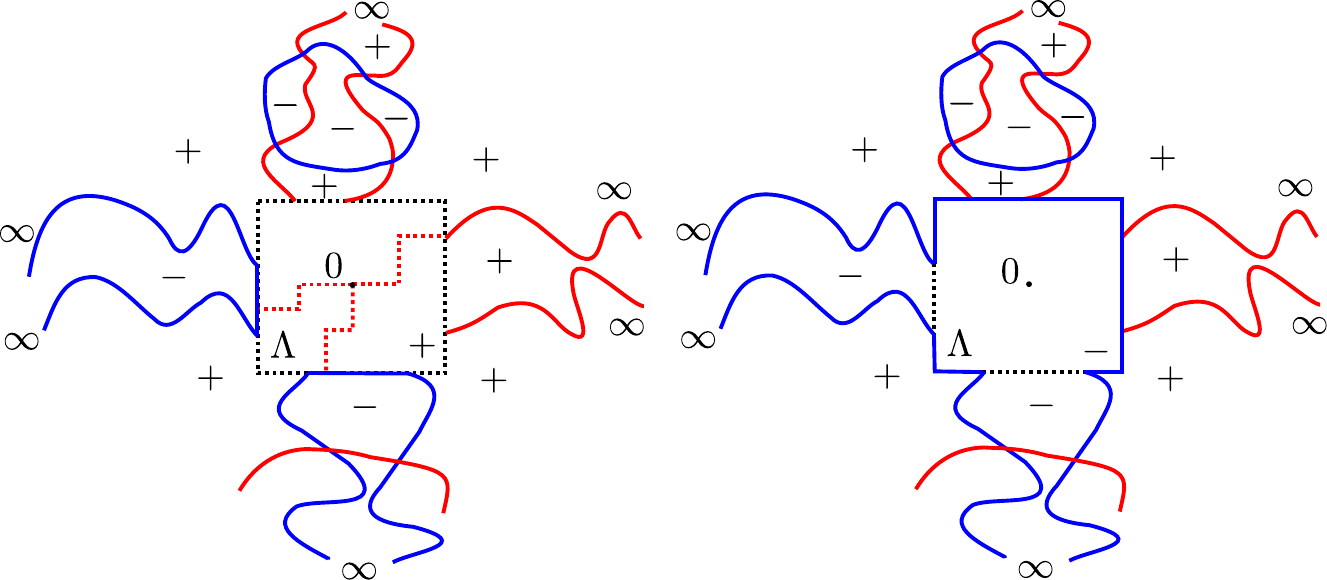}
\caption{The dotted black lines represent the boundary edges of $\Lambda$. 
The signs show the sign of $-i\sigma^{\wi}$, the blue lines represent $\eta(\sigma_{\pm}^\wi)$
and the red lines show the edges of $\omega|_{\Lambda^c} \setminus \eta(\sigma_{\pm}^\wi)$ with $(\sigma,\omega)\in C_6$. There are five infinite clusters in $\omega |_{ \Lambda^c}\cup \eta(\sigma_+^\wi) $ and one in $\omega |_{ \Lambda^c}\cup \eta(\sigma_-^\wi) $.
The red dotted edges are added at finite cost to create a trifurcation at $0$
}
 \label{fig:BK}
\end{center}
\end{figure}
Moreover, we have that
\begin{align*}
T_{\pm}:= \{ (\sigma,\omega)\in S_+(V(\Lambda))\cap C_6: \varsigma (\sigma,\omega) = \pm 1\} \in \mathcal G_{ V(\Lambda^*)^c } \otimes \mathcal F_{\Lambda^c} , 
\end{align*}
where $\mathcal G_{ V(\Lambda^*)^c }$ is the $\sigma$-algebra generated by the spins outside $V(\Lambda^*)$.
Furthermore, for any $I\in \mathcal G_{ V(\Lambda^*)^c } \otimes \mathcal F_{\Lambda^c} $,
\begin{align} \label{eq:delta1}
\mathbf P (S_\pm(V(\Lambda^*)) \mid S_+( V(\Lambda))\cap I ) \geq \delta_1 >0,
\end{align}
where $\delta_1$ depends only on $\Lambda$. 
Indeed, by the definition of the spin model, the only constraint on the values of spins is that if $u,u'\in \mathbb Z^2_{\bl}$ and $v,v'\in \mathbb Z^2_{\wi}$
are incident on a common vertex of $\mathbb Z^2$, then $(\sigma(u)-\sigma(u'))(\sigma(v)-\sigma(v'))=0$. In other words, the interfaces $\eta(\sigma^\bl)$ and $\eta(\sigma^\wi)$ cannot cross.
Since here we assume that $\sigma$ is constant on $V(\Lambda)$, we can always set $\sigma$ to be constant on $V(\Lambda^*)$ and keep this constraint satisfied.
This means that the equivalent of \eqref{eq:delta1} is satisfied by $\mathbf P_\n$ for all $\n$, and hence \eqref{eq:delta1} holds true by taking the weak limit.
We now define 
\begin{align*}
S&= \big(S_+(V(\Lambda^*)) \cup S_-(V(\Lambda^*))\big)\cap S_+(V(\Lambda)), \textnormal{ and }\\
C_3&=\{\textnormal{there are at least \emph{three} infinite clusters in $ \omega|_{ \Lambda^c} \cup \eta( \sigma^\wi)$}\}.
\end{align*}
Note that conditioned on $C_3 \cap S$, one can construct a trifurcation with probability $\delta_2>0$ (depending only on $\Lambda$) by opening some of the edges of $\Lambda$ to create three paths connecting $0$ to three infinite clusters at the boundary of $\Lambda$, and by keeping the remaining edges closed (as depicted on the left-hand side of Fig.~\ref{fig:BK}).
Here we use the definition of the process $\omega$ and fact that $\sigma$ is constant on $V(\Lambda^*)$, and hence the contour configurations
$\eta(\sigma^\wi)$ does not intersect the interior of $\Lambda$.

All in all, we have 
\begin{align*}
\mathbf P(0 \textnormal{ is a trifurcation}) &\geq \delta_2 \mathbf P(C_3 \cap S) \\
& \geq\delta_2[ \mathbf P (S_+(V(\Lambda^*))\cap T_+ ) + \mathbf P (S_-(V(\Lambda^*))\cap T_- )] \\
& =  \delta_2[\mathbf P (S_+(V(\Lambda^*))| T_+ )  \mathbf P (T_+)+\mathbf P (S_-(V(\Lambda^*))| T_- )  \mathbf P (T_-)]\\
&\geq\delta_1\delta_2[  \mathbf P (T_+)+ \mathbf P (T_-)] \\
&= \delta_1\delta_2 \mathbf P(S_+(V(\Lambda))\cap C_6) \\
&= \delta_1\delta_2\mathbf P(S_+(V(\Lambda))| C_6) \mathbf P (C_6)\\
&\geq \delta_1\delta_2 (\tfrac 12) ^{|V( \Lambda)|+1} \\
&>0.
\end{align*}
Using arguments exactly as in~\cite{BK} we finish the proof.
\end{proof}

We are finally ready to show that $\omega$ does not percolate under $\mathbf P$, which by Lemma~\ref{lem:Lis19} will yield delocalization of the height function for $c\in[\sqrt{2+\sqrt 2},2]$.

\begin{proof}[Proof of Proposition~\ref{prop:nopercolation}]
For $u,u'\in \mathbb Z^2_{\bl}$, by Corollary~\ref{cor:P} we have
\begin{align*}
\mathbf P(u \textnormal{ connected to } u' \textnormal{ in } \omega) &= \lim_{N\to \infty} \mathbf P(u \textnormal{ connected to } u' \textnormal{ in } \omega|_{\Lambda_N}) \\
& = \lim_{N\to \infty} \lim_{\n\to \infty}\mathbf P_{\n}(u \textnormal{ connected to } u' \textnormal{ in } \omega|_{\Lambda_N}) \\
&\leq   \lim_{\n\to \infty}\mathbf P_{\n}(u \textnormal{ connected to } u' \textnormal{ in } \omega)\\
& =  \lim_{\n\to \infty} \mathbf{E}_{\mu^0_{\n}}[\sigma(u)\sigma(u')] \\
& = \mathbf{E}_{\mu}[\sigma(u)\sigma(u')].
\end{align*}
The second last equality follows from the Edwards--Sokal property~\eqref{eq:EScor}, and the last one from Proposition~\ref{lem:0conv}.
Combining this with the decorrelation of spins from Theorem~\ref{thm:decorrelation}, we get that 
\begin{align} \label{eq:disconnection}
\mathbf P(u \textnormal{ connected to } u' \textnormal{ in } \omega)\to 0\quad \textnormal{ as } \quad |u-u'|\to \infty.
\end{align}

To finish the proof, we now proceed by contradiction along classical lines. We assume that
$\mathbf P( \exists \textnormal{ an infinite cluster of }\omega) > 0$, and by ergodicity of $\mathbf P$ from Corollary~\ref{cor:P} and Lemma~\ref{lem:BK}, we have that 
\[
\mathbf P(\exists \textnormal{ a unique infinite cluster of } \omega) =1.
\] 
We now fix a box $\Lambda \subset E(\mathbb Z^2_{\bl})$ so large that
\begin{align}\label{eq:34}
\mathbf P(\textnormal{the infinite cluster of } \omega \textnormal{ intersects }\Lambda) \geq 3/4.
\end{align}
Let 
\begin{align*}
A&=\{\textnormal{the infinite cluster of } \omega \textnormal{ intersects }u+\Lambda \textnormal{ and } u'+\Lambda\},\\
B&= \{  \omega \textnormal{ is constant and equal to $1$ on $u+\Lambda$ and $u'+\Lambda$} \}.
\end{align*}
Then, by translation invariance and \eqref{eq:34} we have $\mathbf P(A) \geq 1/2$, and by insertion tolerance from Lemma~\ref{lem:insertion}, 
we have $\mathbf P(B\mid A) \geq \epsilon^{2|\Lambda|} \mathbf P(A)$, where $\epsilon>0$ is as in \eqref{eq:insertion}.
We can now write
\begin{align*}
&\mathbf P(u \textnormal{ connected to } u' \textnormal{ in } \omega ) \geq \mathbf P(A\cap B) \geq \epsilon^{2|\Lambda|} /2.
\end{align*}
Since this lower bound is positive and independent of $u$ and $u'$, we get a contradiction with \eqref{eq:disconnection}, and we finish the proof.
\end{proof}

\bibliographystyle{amsplain}
\bibliography{iloop}
\end{document}